\documentclass[a4paper,11pt]{amsart}

\usepackage{latexsym}
\usepackage{tikz}
\usepackage{amssymb,amsmath}
\usepackage{graphics,mathdots}
\usepackage{url}
\usepackage[vcentermath]{youngtab}
  
\usepackage{booktabs}
\usepackage{blkarray} 
\usepackage{tikz, calc}

\newcommand{\longstep}{\rightsquigarrow}
\newcommand{\rP}[1]{w(#1)}

\newlength\mylength

\def\mylongmapsto#1{%
\begin{tikzpicture}
\draw[line width = 0.75pt] (0,0.5mm) -- (0,-0.5mm);
\setlength{\mylength}{\widthof{#1}}
\draw[line width = 0.75pt, ->] (0,0) -- (1.2\mylength,0) node[above,midway] {#1};
\end{tikzpicture}
}

\newcommand{\dmapsto}{\;\raisebox{0.5pt}{\mylongmapsto{\quad}}}

\newtheorem{theorem}{Theorem}[section]
\newtheorem*{theorem*}{Theorem}
\newtheorem{corollary}[theorem]{Corollary}

\newtheorem{proposition}[theorem]{Proposition}
\newtheorem{conjecture}[theorem]{Conjecture}

\newtheorem{lemma}[theorem]{Lemma}

\numberwithin{equation}{section}

\theoremstyle{definition}
\newtheorem{definition}[theorem]{Definition}
\newtheorem{remark}[theorem]{Remark}
\newtheorem{example}[theorem]{Example}

\newcommand{\N}{\mathbb{N}}

\newcommand{\R}{\mathbb{R}}

\newcommand{\norm}[1]{N(#1)}

\renewcommand{\epsilon}{\varepsilon}

\newcommand{\mbinom}[2]{\left(\!\!\hskip-0.5pt\binom{#1}{#2}\!\!\hskip-0.5pt\right)}
\newcommand{\smbinom}[2]{\bigl(\!\bigl(\genfrac{}{}{0pt}{}{#1}{#2}\bigr)\!\bigr)}

\DeclareMathOperator{\tr}{tr}
\renewcommand{\theta}{\vartheta}

\renewcommand{\d}{\,\mathrm{d}}

\allowdisplaybreaks

\newcommand{\sfrac}[2]{{\scriptstyle\frac{#1}{#2}}}
\newcommand{\mfrac}[2]{{\textstyle\frac{#1}{#2}}}

\newcounter{thmlistcnt}
\newenvironment{thmlist}%
	{\setcounter{thmlistcnt}{0}%
	\begin{list}{\emph{(\roman{thmlistcnt})}}{%
		\usecounter{thmlistcnt}%
		\setlength{\topsep}{0pt}%
		\setlength{\leftmargin}{0pt}%
		\setlength{\itemsep}{0pt}%
		\setlength{\labelwidth}{17pt}
		\setlength{\itemindent}{30pt}}%
	}%
	{\end{list}}%
	
\newcommand{\alphac}[1]{\alpha^{\raisebox{-1pt}{\scalebox{0.95}{$\scriptstyle(#1)$}}}}
\newcommand{\betac}[1]{\beta^{\raisebox{-1pt}{\scalebox{0.95}{$\scriptstyle(#1)$}}}}

\newcommand{\gammai}[1]{\gamma^{\raisebox{-1pt}{\scalebox{0.95}{$\scriptstyle(#1)$}}}}
\newcommand{\gammac}[2]{\gamma^{\raisebox{-1pt}{\scalebox{0.95}{$\scriptstyle(#1,#2)$}}}}
\newcommand{\deltac}[2]{\delta^{\raisebox{-1pt}{\scalebox{0.95}{$\scriptstyle(#1,#2)$}}}}
\newcommand{\kappac}[2]{\kappa^{\raisebox{-1pt}{\scalebox{0.95}{$\scriptstyle(#1,#2)$}}}}

\newcommand{\lambdac}[3]{\lambda^{\raisebox{-1pt}{\scalebox{0.95}{$\scriptstyle(#1,#2)$}}}_{#3}}

\renewcommand{\L}{\mathrm{L}}
\renewcommand{\H}{\mathcal{H}}

\newcommand{\n}{n} 
\newcommand{\DS}{H}
\newcommand{\DSL}{\DS^\lambda}
\newcommand{\PL}{P^\lambda}

\newcommand{\eqkey}{($\star$)}

\newcommand{\Rg}{\mathbb{R}^{> -1}}
\newcommand{\Diag}{\diag}

\renewcommand{\le}{\leqslant}

\renewcommand{\ge}{\geqslant}

\newcommand{\diag}{\mathrm{Diag}}

\renewcommand{\ne}{\{0,\ldots, n-1\}}

\renewcommand{\t}{\mathrm{t}} 
\newcommand{\HM}{L}
\newcommand{\uw}{u}
\newcommand{\FF}{\mathbb{F}}


\begin{document}
\title[]{Involutive random walks on \\ total orders and the anti-diagonal eigenvalue property}
\author{John R.~Britnell and Mark Wildon}
\date{\today}

\begin{abstract}
This paper studies a family of random walks defined on the finite ordinals using their order reversing involutions. Starting at $x \in \{0,1,\ldots,n-1\}$, an element $y \le x$ is chosen according to a prescribed probability distribution, and the walk then steps to $n-1-y$. We show that under very mild assumptions these walks are irreducible, recurrent and ergodic. We then find the invariant distributions, eigenvalues and eigenvectors of a distinguished subfamily of walks whose transition matrices have the global anti-diagonal eigenvalue property studied in earlier work by Ochiai, Sasada, Shirai and Tsuboi. We prove that this subfamily of walks is characterised by their reversibility.
As a corollary, we obtain the invariant distributions and rate of convergence  of the random walk on the set of subsets of $\{1,\ldots, m\}$ in which steps are taken alternately to subsets and supersets, each chosen equiprobably. We then consider analogously defined random walks on the real interval $[0,1]$ and use techniques from the theory of self adjoint compact operators on Hilbert spaces to prove analogues of the main results in the discrete case.
\end{abstract}


\maketitle
\thispagestyle{empty}

\section{Introduction}

In this paper we study a family of random walks defined on the finite ordinals using their
order reversing involutions: starting at $x \in \{0,1,\ldots,n-1\}$, an element $y \le x$ is chosen according to a
prescribed probability distribution, and the walk then steps to $n-1-y$. Under very mild assumptions
we show that these walks are irreducible; if irreducible, 
they are recurrent and ergodic with a unique invariant distribution. 
We then consider a distinguished subfamily of walks whose transition matrices
have the
global anti-diagonal eigenvalue
property studied in earlier work by Ochiai, Sasada, Shirai and Tsuboi \cite{OchiaiEtAl}.
We find their invariant distributions,
eigenvectors and eigenvalues and show that this subfamily is characterised by their reversibility.
We also give a new proof of one of the main results of \cite{OchiaiEtAl} which characterises the generic
(in the sense of algebraic geometry) matrices having the global anti-diagonal eigenvalue property,
and prove a related characterisation that avoids the genericity hypothesis.
As a corollary of our probabilistic results
we obtain the invariant distributions and rate of convergence 
of the random walk on the set of subsets of $\{1,\ldots, m\}$
in which steps are defined by moving to a subset and then to a superset.

In the second part we consider the analogous involutive walks on the real interval $[0,1]$ and use 
techniques from the theory of self-adjoint compact operators on Hilbert spaces to prove
analogues of the main results  in the discrete case. We also prove further results on a
trigonometrically-weighted random walk that do not appear to have discrete analogues
and consider the interplay between the discrete and continuous results.

\subsubsection*{Setup}
Let $[y,x]$ denote the interval $\{y,y+1,\ldots, x\} \subseteq \N_0$ and let $[x]$ denote $\{0,1,\ldots, x\}$.
Let $\star : \{0,1,\ldots,n-1\} \rightarrow \{0,1,\ldots,n-1\}$
be the order anti-involution defined by  $x^\star = n-1-x$; the $n$ will always be clear from context.
It is very convenient to specify the probabilities of steps using the following definition.
 
\begin{definition}\label{defn:weight}
A \emph{weight} with domain $\N_0$ is a function $\gamma$
on the set of intervals of~$\N_0$, taking values in the non-negative real numbers, such that
$\gamma_\varnothing = 0$ and $\sum_{y \in [x]} \gamma_{[y,x]} > 0$ for all $x \in \N_0$.
We say that a weight $\gamma$ is 
\begin{itemize}
\item \emph{strictly positive}
if $\gamma_{[y,x]} > 0$ for all $x$, $y \in n$ with $y \le x$;
\item \emph{atomic} if
$\gamma_{[y,x]} = \gamma_{[y,y]}$ for all $x$, $y \in n$ with $y \le x$; \emph{and}
\item \emph{$\star$-symmetric} if $\gamma_{[y,x]} = \gamma_{[x^\star,y^\star]}$
for all $x$, $y \in n$. 
\end{itemize}
We write $\gamma_y$
for~$\gamma_{[y,y]}$ and $N(\gamma)$ for the strictly positive real-valued function on $\N_0$ defined by
$N(\gamma)_x = \sum_{y \in [x]} \gamma_{[y,x]}$.
\end{definition}

For brevity, when referring to states in an involutive walk, we shall use
the ordinal notation in which $n$ is equal to the subset $\{0,1,\ldots,n-1\}$ of~$\mathbb{N}_0$.
Thus $x \in \{0,1,\ldots, n-1\}$ is written as $x \in n$.
We define a \emph{weight with domain $n$} by replacing the ordinal $\N_0$ with $n$ 
in Definition~\ref{defn:weight}.

\begin{definition}\label{defn:involutiveWalk}
Let $\gamma$ be a weight with domain containing $n$.
The \emph{$\gamma$-weighted involutive walk} on $\n$ 
is the Markov chain with steps defined as follows:
if the current state is $x \in \{0,1,\ldots, n-1\}$ then
choose~$y$ from $[x]$ with probability proportional to $\gamma_{[y,x]}$, and 
step to~$y^\star$.
\end{definition}

Writing $P(\gamma)$ for the transition matrix of the $\gamma$-weighted involutive walk on $n$, Definition~\ref{defn:involutiveWalk}
asserts that
\begin{equation}\label{eq:Pgamma}
P(\gamma)_{xz} = \frac{\gamma_{[z^\star,x]}}{N(\gamma)_x}.
\end{equation}
for all $x$, $z \in n$. (Note that we number rows and columns of matrices from~$0$.)

\subsubsection*{General involutive walks}
Our first main result gives two broadly applicable sufficient conditions for a weighted
involutive walk to be ergodic. 

\begin{samepage}
\begin{theorem}\label{thm:ergodic}
Let $\gamma$ be a weight with domain containing $n$. If \emph{either}
\begin{thmlist}
\item $\gamma_{[0,x]} > 0$ and $\gamma_{[x,n-1]} > 0$ for all $x \in n$; \emph{or}
\item $\gamma_{[x,x]} > 0$ for all $x \in n$ and $\gamma_{[x-1,x]} > 0$ for all non-zero $x \in n$
\end{thmlist}
then the $\gamma$-weighted involutive walk on $n$ is irreducible, recurrent and ergodic
with a unique invariant distribution.
\end{theorem}
\end{samepage}

We are particularly concerned with reversible involutive walks.

\begin{theorem}\label{thm:factorization}
Let $\gamma$ be a weight satisfying either condition in Theorem~\ref{thm:ergodic}. The
$\gamma$-weighted involutive walk on~$\n$ is reversible if and only if $\gamma$ factorizes as a product $\alpha\beta$ where
$\alpha$ is an  atomic weight and $\beta$ is a $\star$-symmetric weight. Moreover, in this case the
unique invariant distribution is proportional to $\alpha_{x^\star}N(\alpha \beta)_x$.
\end{theorem}

\subsubsection*{Binomial weights}
An important family of strictly positive  weights that factorize as in Theorem~\ref{thm:factorization}
is defined for $a$, $b \in \Rg$ by
\begin{equation}\label{eq:gamma} \gammac{a}{b}_{[y,x]} = \binom{y+a}{y} \binom{b+x-y}{x-y}. \end{equation}
(We define binomial coefficients by $\binom{r}{d} = r(r-1) \ldots (r-d+1)/d!$ for all $r \in \R$ and $d \in \N_0$.)
In particular, $\gamma^{(0,0)}$ is the constant weight, and $\gamma^{(1,0)}$ and $\gamma^{(0,1)}$ are 
the weights specifying that $y \in [x]$ is chosen with probability proportional to $\bigl| [y] \bigr| = y+1$ and 
$\bigl|[y,x]\bigr| = x-y+1$, respectively. 

A natural generalization leads to two more families of weights.
Observe that $\smash{\gammac{a}{b}_{[y,x]}}$
is asymptotic to $a^y b^{x-y}/y!(x-y)!$ as $a$, $b \rightarrow \infty$. Setting $b=ac$ and scaling by $a^x$ 
we obtain $c^{x-y}/y!(x-y)!$. We therefore define
$\gammai{c}$ for $c \in \R^{>0}$ by
\begin{equation}\label{eq:gammaInfinity}
\gammai{c}_{[y,x]} = \binom{x}{y}c^{x-y}.\end{equation}
As a further extension, we use that \smash{$\binom{y-a'}{y} = (-1)^y \binom{a'-1}{y}$}
and $\binom{x-y-b'}{x-y} = (-1)^{x-y} \binom{b'-1}{x-y}$ to extend the weights \smash{$\gamma^{(a,b)}$} to $a$, $b \in \R^{< -1}$, by defining 
\begin{equation}\label{eq:delta} \deltac{a'}{b'}_{[y,x]} = \binom{a'-1}{y}\binom{b'-1}{x-y}\end{equation}
for $a',b' > 1$.
If $b' \not\in \N_0$ then $\deltac{a'}{b'}$ is a strictly positive factorizable weight 
with domain
$\min(\lceil a' \rceil, \lceil b' \rceil)$. If $b' \in \N_0$ then $\deltac{a'}{b'}$ is a non-negative
factorizable weight with domain $\lceil a' \rceil$, satisfying hypothesis (ii) in Theorem~\ref{thm:ergodic}. 
(The domains are chosen
to be as large as possible for these positivity properties: see Lemma~\ref{lemma:deltaDomainSimplified}.) In either case,
\smash{$\deltac{a'}{b'}_{[y,x]}$} is equal to
 \smash{$(-1)^x \gammac{-a}{-b}_{[y,x]}$}, when $\gamma^{(-a',-b')}$ 
is defined as a function on intervals in the obvious way.

To orient the reader we show the 
transition matrices $P(\gamma^{(0,0)})$, $P(\gamma^{(1,0)})$, $P(\gamma^{(0,1)})$, $P(\gamma^{(\sfrac{1}{2})})$,
$P(\gamma^{(2)})$ and $P(\delta^{(4,2)})$ below
for $n=4$. 
To emphasise their characteristic \emph{anti-triangular} structure,
we use dots to denote entries that are zero because the relevant interval is empty. This convention is in force throughout.
\[ \hspace*{-1.0in}\scalebox{0.9}{$\displaystyle \begin{matrix} 0 \\ 1 \\ 2 \\ 3 \end{matrix} \ \ 
\left( \begin{matrix} \cdot & \cdot & \cdot & 1 \\ \cdot & \cdot & \mfrac{1}{2} & \mfrac{1}{2} \\
\cdot & \mfrac{1}{3} & \mfrac{1}{3} & \mfrac{1}{3} \\
\mfrac{1}{4} & \mfrac{1}{4} & \mfrac{1}{4} & \mfrac{1}{4} \end{matrix} \right) \quad
\left( \begin{matrix} \cdot & \cdot & \cdot & 1 \\ \cdot & \cdot & \mfrac{2}{3} & \mfrac{1}{3} \\
\cdot & \mfrac{1}{2} & \mfrac{1}{3} & \mfrac{1}{6} \\
\mfrac{2}{5} & \mfrac{3}{10} & \mfrac{1}{5} & \mfrac{1}{10} \end{matrix} \right) \quad
\left( \begin{matrix} \cdot & \cdot & \cdot & 1 \\ \cdot & \cdot & \mfrac{1}{3} & \mfrac{2}{3} \\
\cdot & \mfrac{1}{6} & \mfrac{1}{3} & \mfrac{1}{2} \\
\mfrac{1}{10} & \mfrac{1}{5} & \mfrac{3}{10} & \mfrac{2}{5} \end{matrix} \right) \quad
\left( \begin{matrix} \cdot & \cdot & \cdot &1 \\
\cdot & \cdot & \mfrac{2}{3} & \mfrac{1}{3} \\
\cdot & \mfrac{4}{9} & \mfrac{4}{9} & \mfrac{1}{9} \\
\mfrac{8}{27} & \mfrac{4}{9} & \mfrac{2}{9} & \mfrac{1}{27} \end{matrix} \right)
\quad
\left( \begin{matrix} \cdot & \cdot & \cdot &1 \\
\cdot & \cdot & \mfrac{1}{3} & \mfrac{2}{3} \\
\cdot & \mfrac{1}{9} & \mfrac{4}{9} & \mfrac{4}{9} \\
\mfrac{1}{27} & \mfrac{2}{9} & \mfrac{4}{9} & \mfrac{8}{27} \end{matrix} \right)
\quad
\left( \begin{matrix} \cdot & \cdot & \cdot & 1 \\
\cdot & \cdot & \mfrac{3}{4} & \mfrac{1}{4} \\
\cdot & \mfrac{1}{2} & \mfrac{1}{2} & 0 \\
\mfrac{1}{4} & \mfrac{3}{4} & 0 & 0 \end{matrix} \right).$}
\]
The symmetry under swapping $a$ and $b$ seen by
comparing $P(\gamma^{(1,0)})$, $P(\gamma^{(0,1)})$ becomes a symmetry under swapping
$c$ and $1/c$, seen by comparing $P(\gamma^{(\sfrac{1}{2})})$ and $P(\gamma^{(2)})$.
Generally, if $m \in \N$ and $m < n$ then $P(\delta^{(a',m)})$ has precisely~$m$ non-zero anti-diagonal bands. 

Remarkably, each matrix $P$ above has eigenvalues $(-1)^d P_{dd^\star}$ for $d \in \{0,1,2,3\}$. Up to signs,
these are the entries on its anti-diagonal.
This is generalized in the final claim in the theorem below.

\begin{theorem}\label{thm:spectrum}
The $\gammac{a}{b}$-, $\deltac{a'}{b'}$- and $\gammai{c}$-weighted involutive walks on any~$n$ in their domain are irreducible, recurrent, reversible 
and ergodic. In each case the unique invariant distribution is $\pi$ where~$\pi_x$ is
proportional to $\binom{n-1-x+a}{n-1-x}\binom{x+a+b+1}{x}$ for $\gammac{a}{b}$,
to $\binom{n-1}{x}\bigl( \frac{c+1}{c}\bigr)^x$ for~$\gammai{c}$,
and to $\binom{a'-1}{n-1-x} \binom{(a'-1)+(b'-1)}{x}$ for $\deltac{a'}{b'}$.
The eigenvalues of $P(\gammac{a}{b})$ are 
\begin{equation} \label{eq:evaluesIntroduction} (-1)^d \binom{a+d}{d}/\binom{a+b+d+1}{d}, \end{equation}
the eigenvalues of $P(\gammai{c})$ are $(-1)^d/c^d$, for $d \in \ne$ and
the eigenvalues of $P(\deltac{a'}{b'})$ are 
$(-1)^d \binom{a'-1}{d} / \binom{(a'-1)+(b'-1)}{d}$. In each case, 
up to the signs $(-1)^d$, these are the anti-diagonal entries of the transition matrix.
\end{theorem}

If $a$, $b \in \N_0$ then a more convenient
form for the absolute values of the eigenvalues of $P(\gammac{a}{b})$ is
$\binom{a+b+1}{a}/\binom{a+b+d+1}{b+1}$, in which $d$ appears only once.
We show after~\eqref{eq:chainEigenvalues} that, as this form suggests, these eigenvalues are strictly decreasing
in absolute value.
Hence
the second largest eigenvalue in absolute
value is $-(a+1)/(a+b+2)$. This controls the rate of convergence of
the $\gammac{a}{b}$-weighted
involutive walk. More precisely, by Corollary 12.7 in~\cite{LevinPeres},
\[ \lim_{t \rightarrow \infty}\, \Bigl|\Bigl| \frac{P(\gamma^{(a,b)})^t_{xy}}{\pi_y} - 1 \Bigr|\Bigr|^{1/t} = \frac{a+1}{a+b+1} \]
for any $x \in n$, with similar results for $\gammai{c}$ and $\deltac{a'}{b'}$.

\subsubsection*{The anti-diagonal eigenvalue property}
To motivate our second result and explain the second part of the title, we need some further notation.
If an involutive walk steps from $x$ to $y^\star$, then we say that it makes a \emph{down-step} from $x$ to $y$.
Given a weight $\gamma$, we define the \emph{down-step} matrix $\DS(\gamma)$ 
by $\DS(\gamma)_{xy} = \gamma_{[y,x]}/N(\gamma)_x$. Note that $\DS(\gamma)$ is lower-triangular
and $P(\gamma)_{xz} = \DS(\gamma)_{xz^\star}$. Hence if $J(n)$ is
the $n \times n$ matrix defined, using Iverson bracket notation, by
$J(n)_{xy} = [x=y^\star]$, having ones on its anti-diagonal and zeros elsewhere, then $P(\gamma) = \DS(\gamma) J(n)$.
Theorem~\ref{thm:spectrum} asserts that the eigenvalues of $P(\gamma)$ are, up to signs,
 the eigenvalues of $\DS(\gamma)$, for each  weight~$\gamma$.

More generally, the authors of \cite{OchiaiEtAl} say in their Definition~2.3
that a lower-triangular $n \times n$ matrix~$L$ has the \emph{anti-diagonal eigenvalue property}
if $L$ is diagonalizable and the eigenvalues of $LJ(n)$ are precisely $(-1)^d L_{dd}$ for $d \in \n$. By
Theorem~\ref{thm:spectrum}, the down-step matrices $\DS(\gammac{a}{b})$, $\DS(\deltac{a'}{b'})$ and
$\DS(\gammai{c})$ all have this property.
The authors of~\cite{OchiaiEtAl} conclude after (2.8) that
a complete classification of such matrices
is infeasible. 
They therefore introduce the \emph{global anti-diagonal eigenvalue property},
namely that for all $m \le n$, the $m \times m$ top-left submatrix
of~$H$ has the anti-diagonal property. Since the weights $\gammac{a}{b}$, $\gammai{c}$ and $\deltac{a'}{b'}$
are defined without reference to the size of the matrix, the  down-step matrices 
$\DS(\gammac{a}{b})$,
$\DS(\gammai{c})$ and  $\DS(\deltac{a'}{b'})$ all 
have this stronger property.

One of the main results of~\cite{OchiaiEtAl} is Theorem~2.17, that a
lower-triangular matrix $L$ with eigenvalues
$\lambda_0, \lambda_1, \ldots, \lambda_{n-1}$ read from top to bottom, \emph{such that $(\lambda_0,\lambda_1,\ldots,\lambda_{n-1})$ is
not in an exceptional set of measure zero in $\mathbb{R}^n$}
has the global anti-diagonal eigenvalue property
if and only if $L$ is equal to the matrix  $H^\lambda$ defined by
\begin{equation}\label{eq:HL}
H^\lambda= B(n)\Diag(\lambda_0,\ldots,\lambda_{n-1})B(n)^{-1} \end{equation}
where $B(n)$ is the $n \times n$ Pascal's Triangle matrix defined by
$B_{xy} = \binom{x}{y}$. We say that $H^\lambda$ is a \emph{binomial transform}.
The proof of Theorem~\ref{thm:spectrum} shows that
$H(\gammac{a}{b})$, $H(\deltac{a'}{b'})$, $H(\gammai{c})$ are all binomial transforms
and hence have the global anti-diagonal property.

Since~\cite{OchiaiEtAl} appears to be unpublished, in Theorem~\ref{thm:OchiaiEtAl}
we give a new proof of Theorem~2.17 from~\cite{OchiaiEtAl}  using some basic algebraic geometry to characterize
the exceptional set as the vanishing set of a polynomial.  
Example~\ref{ex:GADEPcx} extends the examples in~\cite{OchiaiEtAl} by exhibiting a family of stochastic matrices
that  have the global anti-diagonal eigenvalue property, but not the generic global anti-diagonal eigenvalue property;
thus the genericity assumption is essential.

As an alternative, we say that an $n \times n$ lower-triangular matrix $L$ with diagonal entries
$\lambda_0, \lambda_1, \ldots, \lambda_{n-1}$ has \emph{global anti-diagonal eigenbasis action} 
if  there is a basis of eigenvectors of $L$ such that $LJ(n)$ has an upper-triangular
matrix in this basis with diagonal entries $\lambda_0, -\lambda_1, \ldots, (-1)^{n-1}\lambda_{n-1}$, and the
same holds 
for all the top-left submatrices of $L$. 
Theorem~\ref{thm:GB} 
characterises matrices with global anti-diagonal eigenbasis action as precisely the binomial
transforms, that is,~those of the form~$H^\lambda$. In particular, $H(\gammac{a}{b})$, $H(\deltac{a'}{b'})$, $H(\gammai{c})$ 
all have global anti-diagonal eigenbasis action.

\subsubsection*{Stochastic matrices with global anti-diagonal eigenbasis action}
Motivated by this observation,
in \S\ref{sec:stochastic} we 
characterize all stochastic matrices that are binomial transforms.
Define $P^\lambda = H^\lambda J(n)$.

\newcounter{bTc}
\setcounter{bTc}{\value{theorem}}
\newcommand{\corollaryBinomialTransform}{Let $\lambda_0, \ldots, \lambda_{n-1} \in \R$. 
The matrix $\PL$ is the transition matrix of an involutive walk if and only if
$\lambda_0 = 1$ and 
\[ \sum_{e=0}^{z} (-1)^e \binom{z}{e} \lambda_{z^\star+e} \ge 0 \] 
for all $z \in n$.
Moreover if strict inequality holds for all $z \in n$ then 
$1 > \lambda_1 > \ldots > \lambda_{n-1}$, and
$\PL_{xz} > 0$ if and only if 
$x+z \ge n-1$.}

\begin{corollary}\label{cor:binomialTransformIsInvolutiveWalk}\label{cor:binomialTransform}
\corollaryBinomialTransform
\end{corollary}

We extend this result in Corollary~\ref{cor:PLergodic} by characterising when the involutive walk defined
by $P^\lambda$ is irreducible, recurrent and ergodic.

It is important to note that \S 2.2 of \cite{OchiaiEtAl} considers the related problem of when a half-infinite
lower-triangular matrix all of whose upper-left submatrices are of the form~\eqref{eq:HL} is stochastic. 
The main result is Theorem 2.34 
which characterises such matrices
as those whose diagonal entries $\lambda_0, \lambda_1, \ldots $
are \emph{completely monotonic}: that is
\[ \sum_{e=0}^k (-1)^e \binom{k}{e} \lambda_{m+e} \ge 0 \]
for all $m$, $k \in \N_0$. The authors
comment at the start of \S 2.2 that `similar results hold in the finite case'. This is true,
but as is clear from Corollary~\ref{cor:binomialTransformIsInvolutiveWalk}, only a one-parameter family of inequalities is required.
Moreover, the inequalities in Corollary~\ref{cor:binomialTransformIsInvolutiveWalk} are specific to each~$n$:  they may
all hold for $\lambda_0, \ldots, \lambda_{n-1}$ and then fail for $\lambda_0, \ldots, \lambda_{n-1}, \lambda_n$.
For instance, this is the case if $\lambda_n$ is large.
Thus it is not possible to obtain our results directly from \cite{OchiaiEtAl};
instead one must adapt the proofs, which require 
the use of the quite deep Hausdorf Moment Theorem.
Another sign that the problems
are somewhat different is that there is no
`limiting case' of the involutive walk, because the ordinal $\N$ has no anti-involution.

The main result in \S\ref{sec:reversibility} further extends Corollary~\ref{cor:binomialTransformIsInvolutiveWalk}.
It is not considered in~\cite{OchiaiEtAl}. We require the following definition.

\begin{definition}\label{defn:globallyReversible}
Let $P$ be the transition matrix of an involutive walk. We say that $P$ is \emph{globally reversible} 
if all the top-right submatrices of $P$ define reversible involutive walks.
\end{definition}

Again, since the weights $\gammac{a}{b}$, $\gammai{c}$ and $\deltac{a'}{b'}$
are defined without reference to the size of the matrix, the transition matrices
$P(\gammac{a}{b})$, $P(\deltac{a'}{b'})$ and
$P(\gammai{c})$ are all globally reversible.
Remarkably, this property characterizes them.

\newcommand{\reversibleImpliesGamma}{Let $n \ge 3$ and suppose that $P^\lambda$ is the transition
matrix of an involutive walk in which $0$
is accessible from every state. 
Then $P^\lambda$ is globally reversible if and only if $P^\lambda = P(\gammac{a}{b})$ for unique $a, b \in \Rg$,
or $P^\lambda = P(\gammai{c})$ for a unique $c \in \R^{>0}$,
or $P^\lambda = P(\deltac{a'}{b'})$ for unique~$a'$, $b' \in \R^{>n-1}$ with either $b' \in \R^{>n-1}$
or $b' \in \N$ and $b' \ge 2$.
}

\begin{theorem}\label{thm:reversibleImpliesGamma}
\reversibleImpliesGamma
\end{theorem}

In fact we prove a somewhat sharper result, stated in Proposition~\ref{prop:reversibleImpliesGamma}, which
implies, for instance,
that if $n=10$ and \smash{$P = P^{(1,\sfrac{2}{3},\nu,\ldots,\lambda_9)}$} 
then $P^\lambda$ is globally reversible if and only if the third largest eigenvalue $\nu$ 
is either in the interval $(\mfrac{1}{3},\mfrac{10}{23})$, or is one of the four exceptional values $\mfrac{10}{23}$, $\mfrac{13}{30}$, $\mfrac{22}{51}$ and~$\mfrac{3}{7}$,
corresponding to the weights $\deltac{17}{9}$, $\deltac{15}{8}$, $\deltac{13}{7}$ and $\deltac{11}{6}$, respectively.
See Example~\ref{ex:reversibleSpectrum} for full details.

We mention that while it is easy to find non-stochastic matrices of the form $\PL$
that satisfy the detailed balanced equations with respect to a $1$-eigenvector, 
the only stochastic examples known to the authors have the
stronger global reversibility property. We therefore make the following conjecture
which we offer as an open problem.

\begin{conjecture}\label{conj:reversible}
An $n \times n$ stochastic matrix of the form $\PL$ is reversible if and only if
 $P^\lambda$ is one of the matrices in Theorem~\ref{thm:reversibleImpliesGamma}, or $P^\lambda = J(n)$.
\end{conjecture}

This conjecture has been verified by computer algebra for $n\le 8$.

\subsubsection*{Continuous walks} The continuous analogue of the discrete involutive walk is defined on the interval $[0,1]$.
Starting at a state $x \in [0,1]$, an element $y \in [0,x]$ is chosen according to a probability 
distribution specified by a real weight (as defined in Definition~\ref{defn:realWeight}); the
walk then steps to $1-y$. 
In \S\ref{sec:Hilbert} and~\S\ref{sec:intervalWalk}
we prove analogues of all the results in \S\ref{sec:factorization} and~\S\ref{sec:gamma}
in the continuous setting. We mention here
that, by Theorem~\ref{thm:ctsLP}, 
the analogue of the $\gammac{a}{b}$-weighted involutive walk 
has discrete spectrum $(-1)^d \binom{a+d}{d}/\binom{a+b+d+1}{d}$ 
for $d \in \N_0$;
suitably shifted Jacobi functions form the corresponding complete orthonormal basis of eigenfunctions
for the analogue of left multiplication by the transition matrix.
Thus the right-eigenvectors for the eigenvalues identified in Theorem~\ref{thm:spectrum}
are the discrete analogue of the Jacobi functions. We believe they are of considerable
combinatorial interest.
We identify some of their properties in~\S\ref{subsec:eigenvectors} and compare then
with the Jacobi functions in the final subsection~\S\ref{subsec:discreteContinuous}.
We also use the extended
example in \S\ref{subsec:trigonometricExample} to show
that one plausible generalization of the anti-diagonal eigenvalue property to the continuous case is false:
the continuous case therefore has a richer theory which again deserves further study.

\subsubsection*{Involutive walks on subsets}
Since $P(\gamma)^2 = H(\gamma) \bigl( J(n)H(\gamma)J(n) \bigr)$, 
taking two steps at a time in a discrete involutive walk gives
the Markov chain in which each step is a down-step followed by an up-step.
An immediate corollary of Theorem~\ref{thm:spectrum} is that, when weighted by $\gammac{a}{b}$ (or one of its generalizations)
this down-up walk on $n$ is also irreducible, reversible,
ergodic~and recurrent. Moreover its eigenvalues
are the squares of the eigenvalues in Theorem~\ref{thm:spectrum}.
A similar result holds on~$[0,1]$ and for the up-down walk. 
Thus despite its more intricate definition, the involutive walk
is the more fundamental of the two random processes. 
This is seen in the corollary~below.

\begin{corollary}\label{cor:subsets}
Fix $m \in \N$ and let $0 < p < 1$. 
Let $\mathcal{P}$ be the set of subsets of $\{1,\ldots, m\}$.
\begin{thmlist}
\item The random walk on $\mathcal{P}$ in which, starting from $X \in \mathcal{P}$,
a subset $Y \subseteq X$ is chosen by putting each $x \in X$ in $Y$ independently with probability~$p$,~and the walk then steps to $\{1,\ldots, m\} \backslash Y$, is irreducible, reversible, recurrent and ergodic
with unique invariant distribution $\pi$ where $\pi_X = p^{m-|X|}/(1+p)^m$ for
each $X \in \mathcal{P}$. Its eigenvalues are $(-p)^e$ for $0 \le e \le m$, with multiplicities~$\binom{m}{e}$.
\item Taking two steps at a time in the walk in (i)
gives the walk on $\mathcal{P}$ in which each step is a down-step followed by an up-step.
This walk has the same invariant distribution, and its eigenvalues are $p^{2e}$ for $0 \le e \le m$, again 
with multiplicities $\binom{m}{e}$.
\end{thmlist}
\end{corollary}

\subsubsection*{A connection with representation theory}
The transition matrix $P(\gamma^{(0,0)})$ shown earlier when $n=4$
appears in
\cite[Example~7.5]{McDowellTensorProducts} in the context of a random walk on
the indecomposable non-projective representations of $\mathrm{SL}_2(\mathbb{F}_p)$ in characteristic~$p$,
as one of two blocks (after rearranging rows and columns) of the matrix shown in Figure~5 of~\cite{McDowellTensorProducts}, taking
the parameter $n$ 
to be $(p-1)/2$.

\subsubsection*{Outline}
In \S\ref{sec:anti-diagonal} we give a self-contained proof of the characterisation~\eqref{eq:HL} of 
matrices with the generic global anti-diagonal property first stated and proved in~\cite{OchiaiEtAl}.
We then prove our sharper characterisation using the new idea of a global anti-diagonal eigenbasis 
introduced in Definition~\ref{defn:GB}. Readers interested mainly in the applications to probability
need only read as far as Proposition~\ref{prop:practical}.
In \S\ref{sec:factorization} we prove Theorems~\ref{thm:ergodic} and~\ref{thm:factorization}. 
In \S\ref{sec:gammaInfinite} we use Proposition~\ref{prop:practical} to
prove Theorem~\ref{thm:spectrum} for the weights~$\gammai{c}$. In~\S\ref{sec:gamma} we extend
the method of \S\ref{sec:gammaInfinite} to prove Theorem~\ref{thm:spectrum} for the weights $\gammac{a}{b}$ and~$\deltac{a'}{b'}$, and give further results
on the left- and right-eigenvectors of the transition matrix $P(\gammac{a}{b})$. We end this section
with the proof of Corollary~\ref{cor:subsets}. In~\S\ref{sec:stochastic} 
we prove 
Corollary~\ref{cor:binomialTransformIsInvolutiveWalk} and in \S\ref{sec:reversibility}
we prove Theorem~\ref{thm:reversibleImpliesGamma}. 
In \S\ref{sec:Hilbert} we 
give a general setting using Hilbert spaces for  involutive walks on the interval $[0,1]$
and in \S\ref{sec:intervalWalk} we prove analogues of the main
results for the $\gammac{a}{b}$-weighted discrete walks, ending by comparing the spectral
behaviour in the discrete and continuous cases.

\section{The global anti-diagonal eigenvalue property}\label{sec:anti-diagonal}

In this section we prove Theorem~\ref{thm:GB} characterizing the binomial transforms 
$H^\lambda$ in~\eqref{eq:HL} in terms of Definition~\ref{defn:GB}.
We also give a self-contained proof of the connection between the binomial transform of a matrix
and the global anti-diagonal eigenvalue property first proved in \cite{OchiaiEtAl}
and prove some simpler results sufficient for the applications to probability below.

\subsection{Preliminaries}
Recall that $B(n)$ is the $n \times n$ Pascal's Triangle matrix defined by $B(n)_{xy} = \binom{x}{y}$
and that $J(n)$ is the $n \times n$ matrix  defined in Iverson bracket notation by $J(n)_{xy} = [x = y^\star]$.
For $e \in \{0,1,\ldots, n\}$, let $V_e$ be the $e$-dimensional subspace of $\mathbb{R}^n$ of all 
row vectors $\bigl( f(0), \ldots, f(n-1) \bigr)$ where~$f$ is a polynomial
of degree strictly less than $e$.
Let 
\begin{equation}\label{eq:v} v(d) = \Bigl( \binom{0}{d}, \binom{1}{d}, \ldots, \binom{n-1}{d} \Bigr)^t. \end{equation}
Thus $B(n)$ has columns $v(0), \ldots, v(n-1)$.

\begin{lemma}{\ }\label{lemma:binomialNew}
\begin{thmlist}
\item $B(n)^{-1}_{zw} = (-1)^{z+w} \binom{z}{w}$;
\item $V_e = \langle v(0), \ldots, v(e-1) \rangle$ for each $e \in \{0,\ldots, n\}$;
\item $J(n)v(d) \in (-1)^d v(d) + \langle v(0), \ldots, v(d-1) \rangle$ for each $d \in \ne$;
\item $\bigl( B(n)^{-1}J(n)B(n) \bigr)_{xy} = (-1)^x \binom{x^\star}{y^\star}$.
\end{thmlist}
\end{lemma}

\begin{proof}
For (i) we calculate 
\begin{align*}
\sum_{w=0}^{n-1}   \binom{x}{w}(-1)^{x+w}B(n)_{wy} &= (-1)^x \sum_{w=0}^{n-1} (-1)^{w} \binom{x}{w} \binom{w}{y} \\ &= (-1)^x \binom{x}{y} \sum_{w=0}^{n-1} (-1)^w \binom{x-y}{w-y} \\ &= [x=y]
\end{align*}
where in the final step we use the corollary of the Binomial Theorem that 
$\sum_{z=0}^m (-1)^z \binom{m}{z} = [m=0]$. 
Since the binomial coefficient $\binom{X}{d}$ is a polynomial with leading term $X^d/ d!$, we have (ii). For (iii) we note that
$J(n)v(d)$ has as its entries $\binom{n-1-X}{d}$ evaluated at $0$, \ldots, $n-1$. Since this
binomial coefficient is a polynomial of degree $d$ with leading term $(-1)^d X^d/d!$,  (iii) follows from~(ii).
Finally for (iv) we apply (i), the identity $\binom{a}{b} = \binom{-b-1}{a-b}(-1)^{a-b}$, Vandermonde's
convolution, and finally $\binom{b-a-1}{b}(-1)^b$ to get
\begin{align*}
\bigl( B(n)^{-1}J(n)B(n) \bigr)_{xy} &= 
\sum_z (-1)^{x+z} \binom{x}{z} \binom{n-1-z}{y} \\
&= \sum_z (-1)^{n-1+x+y} \binom{x}{z} \binom{-y-1}{n-1-z-y} \\
&= (-1)^{n-1+x+y} \binom{x-y-1}{n-1-y} \\
&= (-1)^x \binom{n-1-x}{n-1-y} 
\end{align*}

\vspace*{-3pt}
\noindent as required.
\end{proof}

\enlargethispage{6pt}
\subsection{A sufficient condition for the global anti-diagonal eigenvalue property}

We use the following lemma and proposition in~\S\ref{sec:gammaInfinite} and~\S\ref{sec:gamma} to show that the down-step matrices of the involutive walks for the binomial weights $\gammac{a}{b}$, $\gammai{c}$ and $\deltac{a'}{b'}$ are binomial transforms, and so have the global anti-diagonal eigenvalue property.

\begin{lemma}\label{lemma:practical}
Let $\HM$ be a real lower-triangular matrix with diagonal entries $\lambda_0, \ldots, \lambda_{n-1}$ read from top to bottom.
The following are equivalent
\begin{thmlist}
\item $\HM$ is equal to the binomial transform $H^\lambda$;
\item $\HM = B(n)\mathrm{Diag}(\lambda_0,\ldots,\lambda_{n-1})B(n)^{-1}$;
\item $\HM v(d) = \lambda_d v(d)$ for each $d \in \ne$.
\item $\HM_{xy} = \binom{x}{y}\sum_{e=0}^{x-y} (-1)^e \binom{x-y}{e} \lambda_{y+e}$ for $0 \le y \le x < n$.
\end{thmlist}
\end{lemma}

\begin{proof}
Conditions (i) and (ii) are immediately equivalent by the definition of $H^\lambda$. Since $B(n)$ has columns $v(0), v(1), \ldots, v(n-1)$, (iii) holds if and only if $\HM B(n) = B(n)\mathrm{Diag}(\lambda_0,\ldots,\lambda_{n-1})$, hence if and only if (ii) holds.
Finally, by Lemma~\ref{lemma:binomialNew}(i), 
the entry $H^\lambda_{xy}$ is 
\begin{align*} \sum_{w=0}^{n-1} \binom{x}{w} \lambda_w (-1)^{w+y} \binom{w}{y} &=
\sum_{w=y}^{n-1} \binom{x}{y} (-1)^{w+y} \binom{x-y}{w-y} \lambda_w \\
&= \binom{x}{y} \sum_{e=0}^{x-y} (-1)^e \binom{x-y}{e} \lambda_{y+e}
\end{align*}
as required for (iv).
\end{proof}

\begin{proposition}\label{prop:practical}
Suppose that any of the equivalent conditions in Lemma~\ref{lemma:practical} holds for the $n \times n$ lower-triangular 
matrix $\HM$. Then
\begin{thmlist}
\item $\HM J(n)v(d) \in (-1)^{d}\lambda_d v(d) + V_{d-1}$;
\item $\HM J(n)$ preserves each subspace $V_d$ and acts as  $(-1)^d \lambda_d$ on $V_{d+1}/V_d$;
\item The matrix representing $\HM J(n)$ in the basis $v(0), v(1), \ldots, v(n-1)$ is upper-triangular with diagonal entries
$\lambda_0, -\lambda_1, \ldots, (-1)^{n-1}\lambda_{n-1}$ read from top to bottom.
\item $\HM$ has the global anti-diagonal eigenvalue property.
\end{thmlist}
\end{proposition}

\begin{proof}
By Lemma~\ref{lemma:practical}(iii), $\HM v(d) = \lambda_d v(d)$. By Lemma~\ref{lemma:binomialNew}(ii) and (iii), 
$v(d) \in V_{d+1}$ and $J(n)v(d) = (-1)^d v(d) + w$ where $w \in V_{d}$. Hence $\HM Jv(d) \in (-1)^d \lambda_d v(d) + V_{d}$,
as required for (i). Since $v(d) + V_{d}$ spans the quotient space $V_{d+1}/V_{d}$,~(ii) follows at once;
(iii) simply restates (ii) in matrix language. From~(iii) we see that $\HM J(n)$ has eigenvalues $\lambda_0$, $-\lambda_1$,
\ldots, $(-1)^{n-1}\lambda_{n-1}$ read from top-to-bottom. Hence $\HM$ has the anti-diagonal eigenvalue property.
By applying the canonical projection $\mathbb{R}^n \rightarrow \mathbb{R}^m$ to each subspace $V_d$, 
we see that all these results hold
for each $m \times m$ top-left submatrix of $\HM$.
Hence $\HM$ has the global anti-diagonal eigenvalue property.
\end{proof}

\subsection{A generic characterisation of the global anti-diagonal eigenvalue property}
We now give a self-contained proof of Theorem~2.17 in \cite{OchiaiEtAl}. The matrix
$C$ in our proof plays a similar role to the matrix in Lemma~2.13 in \cite{OchiaiEtAl}, but is more easily visualised.
Stated so as to make clear that the exceptional set is an algebraic variety, Theorem~2.17 is as follows.

\begin{theorem}[Ochiai, Sasada, Shirai and Tsuboi]\label{thm:OchiaiEtAl}
There is a homogeneous polynomial $f(x_0,\ldots,x_{n-2})$ of degree $(n-1)(n-2)/2$ such that,
provided $f(\lambda_0,\ldots,\lambda_{n-2}) \not=0$, a lower-triangular matrix $L$ 
with eigenvalues $\lambda_0, \ldots, \lambda_{n-2}, \lambda_{n-1}$ read from top to bottom
has the global anti-diagonal eigenvalue property
if and only if $L = H^\lambda$.
\end{theorem}

\begin{proof}
The `if' direction holds by Proposition~\ref{prop:practical}(iv).
Suppose that $L$ is a real lower-triangular $n \times n$ matrix with eigenvalues $\lambda_0$, $\ldots$, $\lambda_{n-1}$
having the global anti-diagonal eigenvalue property. 
Working by induction on $n$ we may assume that the matrix obtained from $L$ by deleting its bottom row and rightmost
column is
$H^{(\lambda_0,\ldots,\lambda_{n-2})}$.
Since $H^{(\lambda_0,\ldots,\lambda_{n-2})}$ is obtained by the same deletions from
$H^{(\lambda_0,\ldots,\lambda_{n-2},\lambda_{n-1})}$, we may write
$L = H^\lambda + E$ where the $n \times n$ matrix $E$ is zero except in its bottom row.
Let the bottom row of $E$ be $\epsilon_0, \ldots, \epsilon_{n-1}$ read left to right. 
It suffices to show that $\epsilon_z = 0$ for each~$z$.

By the assumption that $L$ has eigenvalues $\lambda_0, \ldots, \lambda_{n-1}$ we have 
$\tr L = \lambda_0 + \lambda_1 + \cdots + \lambda_{n-1}$. 
Since $\tr L = \tr (H^\lambda + E) = \lambda_0 + \lambda_1 + 
\cdots + \lambda_{n-1} + \epsilon_{n-1}$, it follows that $\epsilon_{n-1} = 0$.
Similarly, since  the matrices $LJ(n)$ and $H^\lambda J(n)$
have eigenvalues $\lambda_0, -\lambda_1, \ldots, (-1)^{n-1}\lambda_{n-1}$, we have $\tr LJ(n) = \tr H^\lambda J(n)$; 
now since $LJ(n) = H^\lambda J(n) + EJ(n)$
and $EJ(n)$ has bottom row entries $\epsilon_{n-1}, \ldots, \epsilon_0$, 
it follows that $\epsilon_0 = 0$.

Let $M = LJ(n) - XI$, where $X$ is an indeterminate.
Since $L$ has the anti-diagonal eigenvalue property,
$\det(LJ-XI) = \det(H^\lambda J(n) - XI)$.
By the multilinearity of the determinant, 
\[ \det(LJ(n)-XI) - \det(H^\lambda J(n) - XI) = \det N \]
where~$N$ is the matrix obtained from $LJ(n) - XI$
by subtracting the bottom row of $H^\lambda J(n)$; this leaves the bottom row of $E$.
For instance if $n=5$ then~$N$ is as shown below.
\[ \scalebox{0.9}{$\left( \begin{matrix}
-X  & \cdot &\cdot &\cdot &  \lambda_0 \\
\cdot & -X & \cdot & \lambda_1 &  \lambda_0-\lambda_1  \\
\cdot & \cdot & \lambda_2-X & 2(\lambda_1-\lambda_2) &  \lambda_0-2\lambda_1 + \lambda_2  \\
\cdot & \lambda_3 & 3(\lambda_2-\lambda_3) & 3(\lambda_0-2\lambda_1 + \lambda_2) - X &  
\lambda_0-3\lambda_1 + 3\lambda_2 - \lambda_3  \\
 0 & \epsilon_3 & \epsilon_2 & \epsilon_1 & 0 \end{matrix} \right)$}  \]
(The entries in $N$ can be computed using Lemma~\ref{lemma:practical}(iv), or more conveniently
for hand calculations, by Lemma~\ref{lemma:binomialTransformRecurrence} below.)

For each $d \in \{0,1,\ldots, n-1\}$, the 
coefficient of $X^d$ in $\det N$ is a linear combination $f_1 \epsilon_1 + \cdots + f_{n-2} \epsilon_{n-2}$
where each $f_y$ is a homogeneous polynomial in $\lambda_0, \ldots, \lambda_{n-2}$ of degree $n-1-d$.
Since the coefficients of $X^0$ and $X^{n-1}$ are zero, we need only consider $d \in \{1,\ldots, n-2\}$.
Let $C$ be the $(n-2) \times (n-2)$ matrix \emph{with rows and columns labelled by $\{1,\ldots, n-2\}$}
 such that $C_{dy}$ is the coefficient of $\epsilon_y X^{d}$ in $\det N$.
There is a non-zero solution $(\epsilon_1, \ldots, \epsilon_{n-2})$ to $\det N = 0$
if and only if $\det C = 0$. We can therefore take the required multivariable
polynomial $f$ to be $\det C$, \emph{provided}
$\det C$ is not the zero polynomial in the eigenvalues $\lambda_0, \ldots, \lambda_{n-1}$.

This is surprisingly non-trivial to prove. 
It suffices to show that $\det C$ is non-zero for one specialization
of the eigenvalues. We take \hbox{$\lambda_x = x$} for each~$x$. (The choice for $\lambda_{n-1}$
is irrelevant.) After this specialization, the matrices $N$ and $C$ are shown left and right below when $n=7$. 
\[ \scalebox{0.8}{$\left(
\begin{matrix}
 -X & \cdot & \cdot& \cdot & \cdot & \cdot & 0 \\
 \cdot & -X & \cdot & \cdot & \cdot & 1 & -1 \\
 \cdot & \cdot & -X & \cdot & 2 & -2 & 0 \\
 \cdot & \cdot & \cdot & 3-X & -3 & 0 & 0 \\
 \cdot & \cdot & 4 & -4 & -X & 0 & 0 \\
 \cdot & 5 & -5 & 0 & 0 & -X & 0 \\
 0 & \epsilon_5 & \epsilon_4 & \epsilon_3 & \epsilon_2 & \epsilon_1 & -X 
\end{matrix}
\right)\quad
\raisebox{-12pt}{$\begin{blockarray}{cccccc}
\begin{block}{c(ccccc)}
X & -120 & -120 & -120 & -120 & -120 \\
X^2 & 100 & 40 & 0 & -30 & -54 \\
X^3 & 15 & 0 & 0 & 10 & 30 \\
X^4 & -5 & 0 & 0 & 0 & 3 \\
X^5 & 0 & 0 & 0 & 0 & -1 \\[-3pt]
\end{block}
& \epsilon_1 & \epsilon_2 & \epsilon_3 & \epsilon_4 & \epsilon_5
\end{blockarray}$}$}
\]
Observe that, after a column permutation $C$ is upper-triangular, with non-zero diagonal entries,
and so has non-zero determinant.

In general, 
to obtain a non-zero contribution to the coefficient
of $\epsilon_y$ by expanding $\det N$ we must take $\epsilon_y$ from the bottom row,
$-X$ from the top row, and $-1$ from row~$1$. Thus $C_{dy}$ is the coefficient
of $\epsilon_ y X^{d-1}$ in the minor obtained by deleting
the first two rows and the first and last columns from~$N$. We may rearrange the rows and
the rows and columns of this minor matrix so that (using the original numbering of $\{0,\ldots, n-1\}$ for both),
the rows come in order $n-2, 2, n-3, 3, \ldots, \lfloor \mfrac{n}{2} \rfloor, n-1$
and the columns in order $1, n-2, 2, n-3, \ldots, \lfloor \mfrac{n}{2} \rfloor$. This gives
an $(n-2) \times (n-2)$ matrix $N'$ that is upper-triangular and tridiagonal, except in its bottom row which has entries
$\epsilon_{n-2}, \epsilon_1, \epsilon_{n-3}, \epsilon_2, \ldots, \epsilon_{\lfloor \sfrac{n-1}{2} \rfloor}$.
(To see this, note that the columns with highest entry in the same row are now adjacent, as are the rows with
left-most entry in the same column.)
By this choice of order, the entries of~$N$ involving $X$ appear on the super-diagonal of $N'$. More precisely,
labelling rows and columns from $0$ as usual, if $x < n-4$ is even then
\begin{align*} N'_{xy} &= \begin{cases} n-2 -\mfrac{x}{2} & \text{if $y=x$} \\
-X & \text{if $y=x+1$} \\
-(n-2 - \mfrac{x}{2}) & \text{if $y=x+2$} \end{cases} 
\intertext{and if $x < n-4$ is odd then}
N'_{xy} &= \begin{cases} -(1 + \mfrac{x+1}{2}) & \text{if $y = x$} \\
-X & \text{if $y=x+1$} \\
1+\mfrac{x+1}{2} & \text{if $y = x+2$} \end{cases}
\end{align*}
and all other entries of $N'$ in these rows are $0$. 
Row $n-4$ of $N'$, just above the bottom row, ends with the two non-zero entries
$\pm \lfloor \mfrac{n}{2} \rfloor$, $-X \mp \lfloor \mfrac{n}{2} \rfloor$, the signs depending on the parity of $n$.
For example, if $n=7$ then $N'$ is as shown below.
\[ \scalebox{0.9}{$\left( \begin{matrix}
 5 & -X & -5 & 0 & 0 \\
 0 & -2 & -X & 2 & 0 \\
 0 & 0 & 4 & -X & -4 \\
 0 & 0 & 0 & -3 & 3-X \\
 \epsilon_5 & \epsilon_1 & \epsilon_4 & \epsilon_2 & \epsilon_3 
\end{matrix} \right)$} \]

The minor matrix obtained by deleting the bottom row and final column from $N'$ is upper-triangular; its determinant is the coefficient of $\epsilon_{\lfloor \sfrac{n-1}{2} \rfloor}$ and has degree $0$ in $X$.
Generally, if $\epsilon_y$ comes $r$-th from the end in the order
$\epsilon_{n-2}$, $\epsilon_1$, $\epsilon_{n-3}$, $\epsilon_2$, \ldots, $\epsilon_{\lfloor \sfrac{n-1}{2} \rfloor}$, 
then the corresponding minor matrix of $N'$
has an $(r-1) \times (r-1)$ bottom right matrix with determinant of degree~\hbox{$r-1$} in $X$ and 
the minor is the product of this by a non-zero constant. Therefore
the polynomial coefficients of 
$\epsilon_{n-2}$, $\epsilon_1$, $\epsilon_{n-3}$, $\epsilon_2$, \ldots, $\epsilon_{\lfloor \sfrac{n-1}{2} \rfloor}$
are of degrees $n-3$, $n-4$, \ldots, $1$, $0$, respectively.
It follows that the matrix $C$ can be rearranged (by the reverse of this ordering of the $\epsilon_y$) so that it is
upper-triangular with non-zero diagonal entries. Hence $\det C \not= 0$.
\end{proof}

The following example extends those in~\cite{OchiaiEtAl} to show that the exceptional set in Theorem~\ref{thm:OchiaiEtAl} is necessary even if we
add the condition that the lower-triangular matrix has distinct eigenvalues, or is stochastic.

\begin{example}\label{ex:GADEPcx}
It is a routine calculation to show that the lower-triangular 
matrices $L$ and $H$ left and middle below both have the anti-diagonal eigenvalue property;
an explicit basis of right-eigenvectors of $HJ(5)$ is given in the far right matrix.
By Lemma~\ref{lemma:practical}(iv), the $3 \times 3$ top-left submatrix of $L$ is 
$H^{(1,\sfrac{2}{3}, \sfrac{1}{4})}$, and similarly, the $4 \times 4$ top-left submatrix
of $H$ is $H^{(1,\sfrac{1}{2},\sfrac{1}{2}-\tau,\sfrac{1}{2}-\tau)}$. Therefore $L$ and $H$ 
have the global anti-diagonal eigenvalue property.
\[ \hspace*{-48pt}\scalebox{0.8}{$
\left( \begin{matrix}
 1 & \cdot & \cdot & \cdot \\
 \frac{1}{3} & \frac{2}{3} & \cdot & \cdot \\[2pt]
 -\frac{1}{12} & \frac{5}{6} & \frac{1}{4} & \cdot \\[2pt]
 -\frac{9}{20} & \frac{11}{10} + \mfrac{4}{5}\tau & \frac{3}{20} + \mfrac{1}{5}\tau & \frac{1}{5} 
\end{matrix} \right) \quad
\left( \begin{matrix} 1 & \cdot & \cdot & \cdot & \cdot \\
\mfrac{1}{2} & \mfrac{1}{2} & \cdot & \cdot & \cdot \\
\mfrac{1}{2} - \tau & 2 \tau & \mfrac{1}{2}-\tau & \cdot & \cdot \\
\mfrac{1}{2}-  2\tau & 3\tau & 0 & \mfrac{1}{2}-\tau & \cdot \\
\mfrac{1}{2}-4\tau & 5\tau & 0 & \tau & \mfrac{1}{2}-2\tau \end{matrix} \right)\quad
\left(\begin{matrix} 
 1 & -2 & 0 & -4 & -4 \tau  (10 \tau -3) \\
 1 & -\frac{5}{4} & 0 & -4 & -28 \tau ^2+11 \tau -1 \\
 1 & -\frac{1}{2} & 1 & -2 \tau -1 & (4 \tau -1) \left(4 \tau ^2-6 \tau +1\right) \\
 1 & \frac{1}{4} & 0 & 2-4 \tau  & (4 \tau -1) \left(8 \tau ^2-5 \tau +1\right) \\
 1 & 1 & 0 & 2-4 \tau  & 2 \tau  (4 \tau -1) (10 \tau -3) 
 \end{matrix}\right)$} \]
Again by Lemma~\ref{lemma:practical}(iv), if $\tau \not=0$ then the bottom row of $L$ does not agree with
$H^{(1,\sfrac{2}{3}, \sfrac{1}{4},\sfrac{1}{5})}$, and similarly the bottom row of $H$ does not agree with
$H^{(1,\sfrac{1}{2},\sfrac{1}{2}-\tau,\sfrac{1}{2}-\tau,\sfrac{1}{2}-2\tau)}$. Therefore these matrices
have the global anti-diagonal eigenvalue property but are not of the form $H^\lambda$.
\end{example}

The authors know of no example satisfying both conditions simultaneously. An explicit calculation of the determinant
$\det C$ in the previous proof can be used to prove the conjecture below for $n \le 4$.

\begin{conjecture}\label{conj:GADEP}
A stochastic matrix with distinct
eigenvalues has the strong anti-diagonal eigenvalue property if and only if it is of the form $H^\lambda$.
\end{conjecture}

We offer this as an open problem, related to Conjecture~\ref{conj:reversible}.

\subsection{Anti-diagonal conjugators and a characterisation of the matrices $H^\lambda$}
Let $L$ be a diagonalizable lower-triangular matrix having diagonal entries $\lambda_0$, \ldots, $\lambda_{n-1}$.
There is a basis $\uw(0), \ldots, \uw(n-1)$ of eigenvectors of $L$ such that the 
matrix $Q$ with columns $\uw(0), \ldots, \uw(n-1)$ is lower-triangular. The matrix representing $LJ(n)$ in this basis
is 
\[ Q^{-1}LJ(n)Q = (Q^{-1}LQ)(Q^{-1}J(n)Q) = \mathrm{Diag}(\lambda_0, \ldots, \lambda_{n-1}) Q^{-1}J(n)Q. \]
Thus if $Q^{-1}J(n)Q$ is upper-triangular with alternating sign entries, then~$L$ has the anti-diagonal eigenvalue
property.
Motivated by this observation we make the following definition,
working over an arbitrary infinite field $\FF$.

\begin{definition}\label{defn:GB}
A lower-triangular $n \times n$ matrix $Q$ with entries in $\FF$ is an \emph{anti-diagonal conjugator} if 
$Q^{-1}J(n)Q$ is upper-triangular with diagonal entries $(Q^{-1}J(n)Q)_{xx} = (-1)^x$ and a \emph{global anti-diagonal
conjugator} if the same holds for all the top-left $m \times m$ submatrices of $Q$, replacing $J(n)$ with $J(m)$.
A lower-triangular matrix $L$ has \emph{global anti-triangular eigenbasis action} if it has the columns of a global
anti-diagonal conjugator as its eigenvectors.
\end{definition}

By Lemma~\ref{lemma:practical}(iv), $B(n)^{-1}J(n)B(n)$ is upper-triangular for each $n$.
Therefore $B(n)$ is a global anti-diagonal conjugator, and the binomial transform $H^\lambda = B(n)\mathrm{Diag}(\lambda_0,\ldots,
\lambda_{n-1})B(n)^{-1}$ has a global anti-diagonal eigenbasis. 
By the following proposition, up to column scaling of $B(n)$, this is the only family of examples.
The final remark in \S\ref{subsec:discreteContinuous} gives an an intuitive interpretation of this proposition.

\begin{proposition}\label{prop:ADC}
A lower-unitriangular $n \times n$ matrix $Q$ is a global anti-diagonal conjugator if and only if $Q = B(n)$.
\end{proposition}

\begin{proof}
We work by induction on $n$. The result is obvious if $n=1$. For the inductive step, we suppose that
$Q$ is an $(n+1) \times (n+1)$ global anti-diagonal conjugator. The $n \times n$ top-left submatrix of $Q$
is a global anti-diagonal conjugator, and so by induction, it is $B(n)$. Let $Q_{nd} = \alpha_d$ for $d \le n$.
Let $\uw(d) \in \FF^{n+1}$ be column $d$ of $Q$ of length $n+1$ and let $v(d) \in \FF^n$ 
be column $d$ of $B(n)$ of length $n$ (as usual).
Thus if $d \in \{0,\ldots,n-1\}$ then $\uw(d)_y = v(d)_y$ if $y < n$ and $\uw(d)_n = \alpha_d$.
By the hypothesis that $J(n+1)Q = QU$ where $U$ is upper-triangular with the specified diagonal entries, we have
\begin{equation}\label{eq:ADCeq} J(n+1)\uw(d) \in (-1)^d \uw(d) + \bigl\langle \uw(0), \ldots, \uw(d-1) \bigr\rangle. \end{equation}

We shall deduce from~\eqref{eq:ADCeq} that $\alpha_d = \binom{n}{d}$.
Since  $J(n+1)\uw(0) = \uw(0)$ we have $\alpha_0 = \uw(0)_n = \uw(0)_0 = 1$. Moreover,
$\alpha_n = Q_{nn} = 1$ since $Q_{nn}$ is uni-triangular. Let $1 \le d < n$ and
let $\mu_0, \ldots, \mu_{d-1}$ be the unique scalars given by~\eqref{eq:ADCeq} such that
\begin{equation}\label{eq:ADCeq2} J(n+1)\uw(d) = (-1)^d \uw(d) + \sum_{x=0}^{d-1} \mu_x \uw(x). \end{equation}
Projecting down to $\FF^n$ by taking the entries in positions $0$, $1$, \ldots, $n-1$ on either side we 
may replace $u(x)$ with $v(x)$ for each $x \le d$ to get
\begin{equation}\label{eq:ADCeq3} \bigl( \alpha_d, \binom{n-1}{d}, \ldots, \binom{1}{d} \bigr)^t = (-1)^d v(d) + \sum_{x=0}^{d-1} \mu_x v(x). \end{equation}
By Lemma~\ref{lemma:binomialNew}(iv), the matrix representing $J(n)$ in the basis $v(0), \ldots v(n-1)$ is
upper-triangular with entries $(-1)^x \binom{n-1-x}{n-1-y}$. Hence
\begin{equation}\label{eq:ADCeq4} J(n)v(d) = (-1)^d v(d) + \sum_{x=0}^{d-1} (-1)^x \binom{n-1-x}{n-1-d} v(x). \end{equation}
This is a somewhat similar relation to~\eqref{eq:ADCeq3}.
Indeed $J(n)v(d)_x = \binom{n-1-x}{d}$, and
adding $\binom{n-1-x}{d-1}$ gives $\binom{n-x}{d}$, 
agreeing with the left-hand side of~\eqref{eq:ADCeq3} for $x \ge 1$.
Observe that 
$\binom{n-1-x}{d-1} = \bigl( J(n)v(d-1)\bigr)_x$ and, by replacing $d$ with $d-1$ in the previous displayed equation,
\begin{equation}\label{eq:ADCeq4b} J(n)v(d-1) = (-1)^{d-1} v(d-1) + \sum_{x=0}^{d-2} (-1)^x \binom{n-1-x}{n-d} v(x). \end{equation}
Adding~\eqref{eq:ADCeq4} and~\eqref{eq:ADCeq4b} and using that $\binom{n-1-x}{d} + \binom{n-1-x}{d-1} = \binom{n-x}{d}$
on the left-hand side and $\binom{n-1-x}{n-1-d} + \binom{n-1-x}{n-d} = \binom{n-x}{n-d}$ on the right-hand side
gives
\begin{align} \nonumber \Bigl( \binom{n}{d}, \binom{n-1}{d}, \ldots, \binom{1}{d} \Bigr)^\t &= 
J(n)v(d) + J(n)v(d-1) \\&= (-1)^d v(d) + \sum_{x=0}^{d-1} (-1)^x \binom{n-x}{n-d} v(x) \label{eq:ADCeq5}. \end{align}
Taking the difference of~\eqref{eq:ADCeq3} and~\eqref{eq:ADCeq5} we get
\[ \Bigl( \alpha_d - \binom{n}{d}, 0, \ldots, 0 \Bigr)^\t = \sum_{x=0}^{d-1} \bigl( \mu_x - (-1)^x \binom{n-x}{n-d} \bigr) v(x). \]
The function $f$ defined by $f(0) = 1$ and $f(x) = 0$ for $x \in \{1,\ldots, n-1\}$ has $n-1$ zeros and is therefore
a polynomial of degree $n-1$. Thus, if non-zero, the left-hand side is in $V_{n}$ but not in $V_e$ for any $e < n$.
On the other hand, the right-hand side is clearly in $V_{d}$. Since $d < n$, we conclude that $\alpha_d = \binom{n}{d}$, as required.
\end{proof}

\begin{theorem}\label{thm:GB} Let $L$ be a lower-triangular $n \times n$ matrix with entries in~$\FF$ having 
diagonal entries $\lambda_0$, $\ldots$, $\lambda_{n-1}$
read from top-bottom. Then $L$ has global anti-diagonal eigenbasis action 
if and only if $H$  is equal to the binomial transform $H^{(\lambda_0,
\ldots, \lambda_{n-1})}$.
\end{theorem}

\begin{proof}
This is immediate from Definition~\ref{defn:GB} and Proposition~\ref{prop:ADC}
\end{proof}

\section{Factorizable weights and proofs of Theorems~\ref{thm:ergodic} and~\ref{thm:factorization}}
\label{sec:factorization}

The second hypothesis in Theorem~\ref{thm:ergodic},
that $\gamma_{[x,x]} > 0$ for all $x \in n$ and $\gamma_{[x-1,x]} > 0$
for all non-zero $x \in n$, implies that the steps $x \mapsto x^\star$ and $x \mapsto (x-1)^\star$ have positive
probability in the $\gamma$-weighted involutive walk for all relevant $x$.
For later use in \S\ref{sec:reversibility}, we take a weaker hypothesis below.

\begin{lemma}\label{lemma:longWalk}
In an involutive walk on $n$, if
$x^\star$ is accessible from $x$ for all $x \in n$ and $(x-1)^\star$ is accessible from $x$ for all non-zero
$x \in n$ then the walk is irreducible.
\end{lemma}

\begin{proof}
The hypothesis implies that, if the involutive walk is allowed to take multiple steps, then each transition in 
\[ 0 \longstep n-1 \longstep 1 \longstep n-2 \longstep 2 \longstep \cdots \longstep n-2 \longstep 1 \longstep n-1 
\longstep 0 \]
has non-zero probability. Therefore all states communicate.
\end{proof}

We are now ready to prove Theorem~\ref{thm:ergodic}.

\begin{proof}[Proof of Theorem~\ref{thm:ergodic}]
If (i) holds then, by~\eqref{eq:Pgamma}, the steps $x \mapsto 0^\star = n-1$ and $n-1 \mapsto x^\star$ have non-zero probability for all $x \in n$.
Hence the involutive walk is irreducible. Moreover, since  $n-1 \mapsto 0^\star = n-1$ is a possible step, it follows from a standard criterion
that the involutive walk is aperiodic. If (ii) holds then, again by~\eqref{eq:Pgamma},
the steps $x \mapsto x^\star$ and $x \mapsto (x-1)^\star$ have non-zero probabilities
and so, by Lemma~\ref{lemma:longWalk}, the involutive walk is again irreducible.
Moreover, if $n=2m$ then $m \mapsto (m-1)^\star = m$ is a possible step and if $n=2m+1$ then $m \mapsto m^\star = m$
is a possible step. Hence the involutive walk is aperiodic.
Thus in either case there is a unique invariant distribution to which the involutive
walk converges by the ergodic theorem.
\end{proof}

We now turn to reversibility.
The detailed balance equations for an invariant probability distribution
$\pi$ on $\n$ are
$\pi_x P(\gamma)_{xz} = \pi_z P(\gamma)_{zx}$. Equivalently,
\begin{equation} \label{eq:db} \pi_x \frac{\gamma_{[z^\star,x]}}{\norm{\gamma}_x}
= \pi_z \frac{\gamma_{[x^\star,z]}}{\norm{\gamma}_z} \end{equation}
 for $x$, $z \in n$.
Since $\gamma_\varnothing = 0$ and $[z^\star,x]$ is non-empty if and only if $[x^\star,z]$ is non-empty,
we may assume that $z^\star \le x$ in this equation. 

\begin{samepage}
\begin{lemma}{\ }\label{lemma:invariant}
\begin{thmlist}
\item If $\alpha$ is an atomic weight then the $\alpha$-weighted involutive walk
is reversible with invariant distribution $\pi(\alpha)$ such that
$\pi(\alpha)_x \propto \alpha_{x^\star} \norm{\alpha}_x$. 
 \item If $\beta$ is a $\star$-symmetric weight then the $\beta$-weighted involutive walk
 is reversible with an invariant distribution $\pi(\beta)$ such that
 $\pi(\beta)_x \propto \norm{\beta}_x$.
\end{thmlist}
\end{lemma}
\end{samepage}

\begin{proof}
When $\alpha$ is atomic the equations~\eqref{eq:db} simplify  to
$\pi_x \alpha_{z^\star} /N(\alpha)_x = \pi_z \alpha_{x^\star} / N(\alpha)_z$
for $x$, $z \in n$ such that $z^\star \le x$. 
Clearly one solution has each $\pi_x$ proportional
to $\alpha_{x^\star} N(\alpha)_x$. 
If $\beta$ is $\star$-symmetric then, by definition, $\beta_{[z^\star,x]} = \beta_{[x^\star,z]}$, and
so one solution
to~\eqref{eq:db} has each $\pi_x$ proportional to~$N(\beta)_x$.
\end{proof}

Since atomic weights are strictly positive, both invariant distributions are strictly positive.
Reversibility is preserved by products of weights.

\begin{lemma}\label{lemma:pointwiseProductReversible}
Let $\alpha$ and $\beta$ be  weights. 
If the $\alpha$-weighted and $\beta$-weighted involutive walks on $n$
are reversible with respect to invariant distributions 
proportional to $\theta$ and $\phi$, respectively,
then the $\alpha\beta$-weighted involutive walk on $n$ is reversible
with respect to an invariant distribution proportional to $\theta_x \phi_x N(\alpha\beta)_x
/N(\alpha)_xN(\beta)_x$.
\end{lemma}

\begin{proof}
Multiplying the two cases of~\eqref{eq:db} for $\alpha$ and $\beta$ we obtain
\[ \theta_x \phi_x \frac{\alpha_{[z^\star,x]}\beta_{[z^\star,x]}}{\norm{\alpha}_x \norm{\beta}_x}  = 
\theta_z \phi_z \frac{\alpha_{[x^\star,z]}\beta_{[x^\star,z]}}{\norm{\alpha}_z \norm{\beta}_z}
\]
for all $x$, $z \in n$. Therefore $\pi_x = \theta_x \phi_x \norm{\alpha\beta}_x / \norm{\alpha}_x\norm{\beta}_x$ solves~\eqref{eq:db} for the weight $\alpha\beta$.
\end{proof}

We are now ready to prove Theorem~\ref{thm:factorization}.

\begin{proof}[Proof of Theorem~\ref{thm:factorization}]
Suppose that $\gamma = \alpha\beta$ factorizes as in the theorem. By Lemma~\ref{lemma:invariant},
the $\alpha$- and $\beta$-weighted involutive walks are reversible with respect to invariant distributions proportional to $\alpha_{x^\star} N(\alpha)_x$ and $N(\beta)_x$, respectively. Therefore, by Lemma~\ref{lemma:pointwiseProductReversible}, the $\gamma$-weighted involutive walk is reversible with
respect to an invariant distribution proportional to 
\[ \alpha_{x^\star} N(\alpha)_x N(\beta)_x \times \frac{N(\alpha\beta)_x}{N(\alpha)_xN(\beta)_x} =
\alpha_{x^\star} N(\alpha\beta)_x. \]
By Theorem~\ref{thm:ergodic}, this invariant distribution is unique.

Suppose conversely that the $\gamma$-weighted involutive walk is reversible
with respect to an invariant distribution $\pi$. By the hypothesis of Theorem~\ref{thm:factorization} and Theorem~\ref{thm:ergodic},
the chain is irreducible and ergodic, and so this invariant distribution is strictly positive.
We must define an atomic weight $\alpha$ so that $\beta$, defined by $\beta = \gamma/\alpha$,
is $\star$-symmetric.
Taking $y = z^\star$ in~\eqref{eq:db}
and rearranging we obtain
\begin{equation}\label{eq:db2} 
\frac{N(\gamma)_{y^\star}}{\pi_{y^\star}} \gamma_{[y,x]}= \frac{N(\gamma)_{x^\star}}{\pi_{x^\star}} \gamma_{[x^\star,y^\star]} 
\end{equation}
for all $x$, $y \in \n$. 
Noting that the fractions on each side depend only on $y^\star$ and only on $x^\star$, respectively, we 
are led to define an atomic weight $\alpha$ by $\alpha_{[y,x]} = \frac{\pi_{y^\star}}{N(\gamma)_{y^\star}}$.
Now by~\eqref{eq:db2},
\[ \beta_{[y,x]} = \frac{\gamma_{[y,x]}}{\alpha_y} = \frac{N(\gamma)_{y^\star}}{\pi_{y^\star}}\gamma_{[y,x]}
=  \frac{N(\gamma)_{x^\star}}{\pi_{x^\star}} \gamma_{[x^\star,y^\star]} = \frac{\gamma_{[x^\star,y^\star]}}{\alpha_{y^\star}}
= \beta_{[x^\star,y^\star]} \]
so $\beta$ is $\star$-symmetric, as required.
\end{proof}

\section{The $\protect\gammai{c}$-weighted involutive walk}\label{sec:gammaInfinityWalk}
\label{sec:gammaInfinite} 

Let $c \in \R^{>0}$. Recall that $\gammai{c}$ is the weight defined by $\gammai{c}_{[y,x]} = \binom{x}{y}c^{x-y}$
for $y \le x$.
Since $N(\gammai{c})_x = (c+1)^x$ we have 
\begin{equation}\label{eq:Hgammai} H(\gammai{c})_{xy}
= \binom{x}{y}\frac{c^{x-y}}{(c+1)^x}.\end{equation} 
In particular, $H(\gamma^{(1)})$ is the Pascal's Triangle matrix with rows scaled so
that it is stochastic.
Observe that the eigenvalues of $H(\gammai{c})$, and so the anti-diagonal entries of $P(\gammai{c})$, 
are $1/(c+1)^d$ for $d \in \ne$.

\subsection{Stochastic properties of $P(\protect\gammai{c})$}\label{subsec:gammaInfiniteFactorization}
Since $\gammai{c}$ is strictly positive, Theorem~\ref{thm:ergodic} implies
that the $\gammai{c}$-weighted involutive walk is irreducible, recurrent and ergodic
with a unique invariant distribution.
To find the invariant distribution we first show that 
$\gammai{c}$ is factorizable. (This is motivated by the interpretation of $\gammai{c}$ as a limiting
case of the clearly factorizable weights $\gammac{a}{ac}$.) 
From the required form $\binom{x}{y}c^{x-y} 
= \alpha_y \beta_{[y,x]}$, 
where $\beta$ is $\star$-symmetric, we must define $\alpha_y$ so that $\alpha_yx!c^{x-y} / y!(x-y)! $ is
invariant under $(x,y) \mapsto (y^\star,x^\star)$. This already holds for $c^{x-y}$ and $(x-y)!$, and similarly
to the proof of Theorem~\ref{thm:factorization}, we may cancel $y!$
and introduce the necessary factor of $y^\star!$ by taking
$\alpha_y$ proportional to $1/y!(n-1-y)!$. We therefore
define $\alpha_y = \binom{n-1}{y}$ as a weight with domain~$n$, and find that
\begin{equation}\label{eq:betai} \beta_{[y,x]} = \frac{\binom{x}{y}c^{x-y}}{\binom{n-1}{y}} = \frac{x!y^\star!}{(x-y)!(n-1)!}c^{x-y} \end{equation}
is indeed $\star$-symmetric. 
Theorem~\ref{thm:factorization} now implies that the
$\gammai{c}$-weighted involutive walk on $n$ is reversible
with unique invariant distribution proportional to $\alpha_{x^\star}N(\gammai{c})_x$; that is
\[ \pi_x \propto \binom{n-1}{x} \Bigl( \frac{c+1}{c} \Bigr)^x .\]
Explicitly, $\displaystyle \pi_x = \binom{n-1}{x} \frac{(c+1)^x c^{n-1-x}}{(2c+1)^{n-1}}$.

\subsection{Spectrum of $\protect\gammai{c}$}
We now prove
that the eigenvalues of $P(\gammai{c})$ are
$(-1)^d/c^d$ for $d \in \ne$ using the condition in Lemma~\ref{lemma:practical}(iii) 
for $H(\gammai{c})$ to be a binomial transform.
Recall from~\eqref{eq:v} that $v(d) \in \R^n$ is the column vector with entries \hbox{$v(d)_x = \binom{x}{d}$}.

\begin{lemma}Let $c \in \R^{>0}$. Then \label{lemma:gammaInfiniteEigenvectors}
\[ H(\gammai{c}) v(d) = v(d)/(c+1)^d. \]
\end{lemma}

\begin{proof}
Since $H(\gammai{c})_{xy} = \binom{x}{y} \frac{c^{x-y}}{(c+1)^x}$ we have
\[ \begin{split} 
\bigl(H(\gammai{c})v(d)\bigr)_x &= \sum_{y=0}^{n-1} \binom{x}{y}\frac{c^{x-y}}{(c+1)^x} \binom{y}{d} 
= \binom{x}{d} \frac{c^{x-d}}{(c+1)^x} \sum_{y=0}^{n-1} \binom{x-d}{y-d} \frac{1}{c^{y-d}} \\ &\qquad =
\binom{x}{d} \frac{c^{x-d}}{(c+1)^x} \Bigl( 1+\frac{1}{c}\Bigr)^{x-d} = \binom{x}{d}\frac{1}{(c+1)^d}
= \frac{v(d)_x}{(c+1)^d}
 \end{split} \]
as required.
\end{proof}

Hence $H(\gammai{c})$ is the binomial transform $H^{\lambda^{(c)}}$, where $\lambda^{(c)}_d = 1/(c+1)^d$
for each $d \in \ne$, and, by Proposition~\ref{prop:practical}(iii), $P(\gammai{c})$ has 
eigenvalues $(-1)^d/(c+1)^d$ for $d \in \ne$, as required.
This completes the proof of Theorem~\ref{thm:spectrum} for $\gammai{c}$.

\section{The $\protect\gammac{a}{b}$-weighted and $\protect\deltac{a'}{b'}$-weighted involutive walks}\label{sec:gamma}

We now adopt a similar strategy to the previous section to prove Theorem~\ref{thm:spectrum} for 
the $\gammac{a}{b}$ and $\deltac{a'}{b'}$ involutive walks.
It is convenient in this section to extend this definition of $\gammac{a}{b}$ 
to define $\gammac{a}{b}_{[y,x]} = \binom{a+y}{y} \binom{b+x-y}{x-y}$ for all $a$, $b \in \R$
and all $x$, $y \in \N$ with $y \le x$.

\subsection{Preliminaries on multisubsets}\label{subsec:multisubsets}
A combinatorial interpretation of $\gammac{a}{b}$ is helpful and motivates many steps below. Recall that 
$\binom{m+c-1}{c}$, or equally $(-1)^c \binom{-m}{c}$, is the number of
$c$-multisubsets of a set of size $m$, or equivalently, the number of chains
$0 \le y_1 \le \ldots \le y_c \le m-1$ in the ordinal $m$. Thus when $a, b \in \N_0$, 
\[ \scalebox{0.925}{$\displaystyle 
\gammac{a}{b}_{[y,x]} = \bigl| \bigl\{ (y_1,\ldots, y_a, w_1, \ldots, w_b) : 0 \le y_1\le \ldots \le y_a \le y \le w_1 \le \ldots \le w_b \le x \bigr\} \bigr|$.} \]

We write $\smbinom{m}{c}$ for $\binom{m+c-1}{c}$. Note that 
$\smbinom{m}{c} = \binom{m+c-1}{c} = \binom{m+c-1}{m-1} = \smbinom{c+1}{m-1}$, and
so, unlike
the normal binomial coefficient, $\smbinom{m}{c}$ is, separately, a polynomial function of both $m$ and $c$.
We use this fact repeatedly below to deduce results on $\gammac{a}{b}$ for general $a$, $b \in \R$
from the special cases when $a$, $b \in \N_0$. 

\begin{lemma}\label{lemma:mbinom}
Let $m$, $n \in \N$.
\begin{thmlist}
\item If $c \in \N_0$ then $\sum_{y=0}^x \smbinom{y+1}{c} = \smbinom{x+1}{c+1}$.
\item If $c$, $d \in \N_0$ then $\sum_{x=0}^{n-1} \smbinom{n-x}{c}\smbinom{x+1}{d} = \smbinom{n}{c+d+1}$.
\item If $c$, $d \in \N_0$ then
$\sum_{y=0}^x \smbinom{x-y+1}{c}\binom{y}{d} = \binom{x+c+1}{c+d+1}$.
\item If $c \in \N_0$ then
$\sum_{y=0}^x \smbinom{x-y+1}{c} (-1)^y \binom{m}{y} = (-1)^x \binom{m-c-1}{x}$.
\end{thmlist}
\end{lemma}

\begin{proof}
For (i) we count $(c+1)$-multisubsets of $\{0,\ldots, x\}$:
the $(c+1)$-multisubsets having $y$ as their greatest element
are in bijection with the
$c$-multisubsets of $\{0,\ldots, y\}$; these are counted by $\smbinom{y+1}{c}$. 
For (ii), use $\smbinom{n-x}{c} = \smbinom{c+1}{n-x-1}$
to rewrite the left-hand side as
$\sum_{x=0}^{n-1} \smbinom{c+1}{n-x-1}\smbinom{d+1}{x}$.
This counts the $(n-1)$-multisubsets of the disjoint union
of sets of size $c+1$ and $d+1$, and so is $\smbinom{c+d+2}{n-1} = \smbinom{n}{c+d+1}$.
For (iii), the non-zero summands occur for $y\ge d$.
We rewrite $\binom{y}{d}$ as $\smbinom{y-d+1}{d}$ and apply (ii)
to the right-hand side~of
\[ \sum_{y=d}^x \mbinom{x-y+1}{c}
\mbinom{y-d+1}{d} = \sum_{z=0}^{x-d} \mbinom{x-d+1-z}{c} \mbinom{z+1}{d} \]
to get $\smbinom{x-d+1}{c+d+1}$, which is $\binom{x+c+1}{c+d+1}$, as required.
Finally for (iv), we use $\smbinom{x-y+1}{c} = \binom{x-y+c}{c} = (-1)^{x-y}\binom{-c-1}{x-y}$
to rewrite the right-hand side as \smash{$(-1)^x \sum_{y=0}^x \binom{-c-1}{x-y} \binom{m}{y}$};
by Vandermonde's convolution this is \smash{$(-1)^x \binom{m-c-1}{x}$}, as required.
\end{proof}

\subsection{Stochastic properties of  $P(\protect\gammac{a}{b})$ and $P(\protect\deltac{a'}{b'})$}\label{subsec:gammaMatrices}
The function $\gammac{a}{b}$ factorizes 
as the product $\alpha^{(a)}\beta^{(b)}$ where
$\alpha^{(a)}$ is defined by \smash{$\alpha^{(a)}_y = \binom{y+a}{y}$} and 
 $\beta^{(b)}$ is  defined by
\smash{$\beta^{(b)}_{[y,x]} = \binom{x-y+b}{x-y}$}. When $a$, $b \in \Rg$, $\alpha^{(a)}$ is a strictly positive
atomic weight and $\beta^{(b)}$ is a strictly positive $\star$-symmetric weight.
The weight $\deltac{a'}{b'} = \binom{a'-1}{y}\binom{b'-1}{x-y}$ defined for $a', b' > 1$ has a similar factorization.
If $b' \not\in \N$ then the domain of $\deltac{a'}{b'}$ 
is $\min(\lceil a' \rceil, \lceil b' \rceil)$ and $\deltac{a'}{b'}$
is strictly positive. 
Otherwise $\deltac{a'}{b'}$ has the potentially larger domain $\lceil a' \rceil$
and \smash{$\deltac{a'}{b'}_{[y,x]} > 0$} if and only if $y-x < b'$. Since $b' \ge 2$,
the hypothesis of Theorem~\ref{thm:ergodic} holds. Therefore in all cases
the involutive walk is irreducible, recurrent and ergodic with a unique invariant distribution.

To find the invariant invariant distributions we use Theorem~\ref{thm:factorization}.
The following lemma is required.


\begin{lemma}\label{lemma:chainWeightsTotal}
Let $x$, $y \in \n$  with $y \le x$. For $a$, $b \in \R$ we have
\begin{thmlist}
\item $\alphac{a}_{[y,x]} = \smbinom{y+1}{a}$ and $\norm{\alphac{a}}_x = \smbinom{x+1}{a+1}$;
\item $\betac{b}_{[y,x]} = \smbinom{x-y+1}{b}$ and $\norm{\betac{b}}_x = \smbinom{x+1}{b+1}$;
\item $\gammac{a}{b}_{[y,x]} = \smbinom{y+1}{a} \smbinom{x-y+1}{b}$ and $\norm{\gammac{a}{b}}_x = 
\smbinom{x+1}{a+b+1}$.
\end{thmlist}
\end{lemma}

\begin{proof}
The first parts of (i), (ii) and (iii) are immediate from the definitions. Since 
$N(\alpha^{(a)})_x$, $N(\beta^{(b)})_x$ and $N(\gammac{a}{b}_x$ are polynomials in $a$ and $b$,
it suffices 
to prove the remaining results when $a$ and $b$ are integral. 
In this case the normalization factors $N(\alpha^{(a)})_x$ and $N(\beta^{(b)})_x$ 
follow immediately from Lemma~\ref{lemma:mbinom}(i).
For $\gammac{a}{b}$ we have
\[ N(\gammac{a}{b})_x = \sum_{y=0}^x \gammac{a}{b}_{[y,x]} 
= \sum_{y = 0}^x \mbinom{y+1}{a}\mbinom{x-y+1}{b} 
= \mbinom{x+1}{a+b+1} \]
by Lemma~\ref{lemma:mbinom}(ii).
\end{proof}

Thus $N(\gammac{a}{b})_x  = \binom{x+a+b+1}{x}$ and so $N(\gammac{a}{b})_x > 0$ for $x \in n$ if and only if 
$a+b+2$, \ldots, $a+b+n \not=0$, or equivalently,
\begin{equation}\label{eq:HgammaCondition}
a+b \not\in \{-2,\ldots, -n\}.
\end{equation}
We saw after~\eqref{eq:delta} in the introduction that \smash{$\deltac{a'}{b'}_{[y,x]} = (-1)^x \gammac{-a'}{-b'}_{[y,x]}$} for $a'$, $b' \in \R^{> 1}$.
Hence
$N(\deltac{a'}{b'})_x = (-1)^xN(\gammac{-a}{-b})$. Using that $\smbinom{m}{c} = \smbinom{c+1}{m-1}$,
it follows from Lemma~\ref{lemma:chainWeightsTotal}(iii) that
\begin{equation}\label{eq:normDelta} N(\deltac{a'}{b'})_x = \binom{(a'-1)+(b'-1)}{x}. \end{equation}
Hence as expected from~\eqref{eq:HgammaCondition}, $N(\deltac{a'}{b'})_x > 0$
whenever $x \in n$ and $n$ is in the domain of $\deltac{a'}{b'}$.

By Theorem~\ref{thm:factorization}, if $a$, $b \in \R^{>-1}$, the unique invariant
distribution for the $\gammac{a}{b}$-weighted involutive walk on $n$ 
is proportional to 
$\alphac{a}_{n-1-x} N(\gammac{a}{b})_x$. Therefore, by Lemma~\ref{lemma:chainWeightsTotal}, it is $\pi$ where
\begin{equation}\label{eq:gammaInvariant} \pi_x \propto \mbinom{n-x}{a} \mbinom{x+1}{a+b+1}. \end{equation}
The normalization factor for $\pi$ is $\smbinom{n}{2a+b+2}$ by Lemma~\ref{lemma:mbinom}(ii).
Similarly for $\deltac{a'}{b'}$, by~\eqref{eq:normDelta}, 
the unique invariant distribution on $n$, for $n$ in the domain of $\deltac{a'}{b'}$
is $\phi$, where
\begin{equation}\label{eq:deltaInvariant}
\phi_x \propto \binom{a'-1}{x} \binom{(a'-1)+(b'-1)}{x}. 
\end{equation}
Rewriting the first binomial coefficient
as \smash{$\binom{a'-1}{a'-1-x}$} and applying Vandermonde's convolution, we see that the normalization
factor for $\phi$ is \smash{$\binom{2a'+b'-3}{n-1}$}. (This may also be obtained by substituting $-a'$ for $a$ and $-b'$ for $b$
in the normalization factor for $\pi$.) 

\subsection{Spectrum of $P(\protect\gammac{a}{b})$ and $P(\protect\deltac{a'}{b'})$}\label{sec:gammaSpectrum}
To complete the proof of Theorem~\ref{thm:spectrum}, we must show that the eigenvalues of these transition
matrices are as claimed.
Again we deal with 
the two families in a uniform way using
that 
\begin{equation}\label{eq:HgammaDelta} H(\gammac{-a'}{-b'}) = H(\deltac{a'}{b'}).\end{equation}
Whenever~\eqref{eq:HgammaCondition} holds,
the eigenvalues of $H(\gammac{a}{b})$ are its diagonal entries $\gammac{a}{b}_{[x,x}]/N(\gammac{a}{b})_x$.
By  Lemma~\ref{lemma:chainWeightsTotal}(iii),
these are  $\lambda_0^{(a,b)}, \ldots, \lambda_{n-1}^{(a,b)}$ where
\begin{equation}\label{eq:chainEigenvalues} \lambdac{a}{b}{d} = \mbinom{d+1}{a} / \mbinom{d+1}{a+b+1}. \end{equation}
The following lemma, analogous to Lemma~\ref{lemma:gammaInfiniteEigenvectors}, reveals the eigenvectors and,
as in \S\ref{sec:gammaInfinite},
verifies the condition in Lemma~\ref{lemma:practical}(iii) for $H(\gammac{a}{b})$ to be a binomial transform.
Recall from~\eqref{eq:v} that $v(d) \in \R^n$ is the column vector with entries \hbox{$v(d)_x = \binom{x}{d}$}.

\begin{lemma}\label{lemma:Htransform}
Let $a$, $b \in \R$ and let $n \in \N$ be such that $a+b \not\in \{-2,\ldots, -n\}$. Then
\[  H(\gammac{a}{b}) v(d)  = \frac{\smbinom{d+1}{a}}{\smbinom{d+1}{a+b+1}} v(d). \]
\end{lemma}

\begin{proof}
We have
\begin{equation}\label{eq:Ha} \bigl( H(\gammac{a}{b}) v(d) \bigr)_x 
=\sum_{y=0}^x \mbinom{y+1}{a}\!\mbinom{x-y+1}{b}\!\binom{y}{d} / \!\mbinom{x+1}{a+b+1}\!. \end{equation}
When $a$, $b \in \mathbb{N}_0$, by $\binom{y+a}{a}\binom{y}{d} = \binom{a+d}{a}\binom{y+a}{a+d}$ 
 and an instance
of Lemma~\ref{lemma:mbinom}(iii), the numerator on the right-hand side is
\begin{align*}
\sum_{y=0}^x \binom{y+a}{a}\binom{y}{d} \mbinom{x-y+1}{b} &= \binom{a+d}{a}
\sum_{y=0}^x \binom{y+a}{a+d}\mbinom{x-y+1}{b} \\
& = \binom{a+d}{a}  \binom{x+a+b+1}{a+b+d+1}.
\end{align*}
We now use
$\binom{x+a+b+1}{a+b+d+1} / \smbinom{x+1}{a+b+1} = \binom{x+a+b+1}{a+b+d+1} / \binom{x+a+b+1}{a+b+1}
= \binom{x}{d} / \binom{a+b+d+1}{a+b+1}$
to find that the right-hand side of~\eqref{eq:Ha} is
\begin{align*} 
\binom{a+d}{a} \binom{x}{d} / \binom{a+b+d+1}{a+b+1} &= \binom{x}{d} \mbinom{d+1}{a} / \mbinom{d+1}{a+b+1} \\
&= v(d)_x \mbinom{d+1}{a} / \mbinom{d+1}{a+b+1}. \end{align*}
This proves the equality claimed in the lemma when $a$, $b \in \mathbb{N}_0$, and since,
by multiplying through by the non-zero quantities
$N(\gammac{a}{b})_x$ and $\smbinom{d+1}{a+b+1}$ each side becomes
polynomial in $a$ and $b$, it holds for all claimed~$a$ and $b$.
\end{proof}

Hence $H(\gammac{a}{b})$ is the binomial transform $H^{\lambda^{(a,b)}}$ where $\lambda^{(a,b)}_d$
is as defined in~\eqref{eq:chainEigenvalues}. By Proposition~\ref{prop:practical}(iii), $P(\gammac{a}{b})$
and $P(\deltac{a'}{b'})$ have the required eigenvalues.
This completes the proof of Theorem~\ref{thm:spectrum} for $\gammac{a}{b}$ and $\deltac{a'}{b'}$.

\subsubsection*{Second largest eigenvalue}
Since
\[ \frac{\lambdac{a}{b}{d}}{\lambdac{a}{b}{d+1}} = 
\frac{\binom{d+a-1}{d-1} / \binom{d+a+b}{d-1}}{\binom{d+a}{d} / \binom{d+a+b+1}{d}} 
= \frac{d}{d+a} \frac{d+a+b+1}{d} = \frac{a+b+1+d}{a+d} \]
the sequence of eigenvalues of $P(\gammac{a}{b})$ is decreasing in absolute value.
The same holds for $P(\deltac{a'}{b'})$ rewriting the right-hand side as $(a'+b'-d-1)/(a'-d)$.
The ratio above also implies that $\lambdac{a}{b}{d} \rightarrow 0$ as $d \rightarrow \infty$.
The rate of convergence of the involutive walk is controlled by the second
largest eigenvalue (in absolute value) of $P(\gammac{a}{b})$, namely $\lambdac{a}{b}{1} = (a+1)/(a+b+2)$;
for $P(\deltac{a}{b})$, the rewriting $(a'-1)/\bigl((a'-1)+(b'-1)\bigr)$ is most convenient.

\subsection{Eigenvectors of $P(\protect\gammac{a}{b})$}\label{subsec:eigenvectors}
We now explore the spectral behaviour of $P(\gammac{a}{b})$ when $a$, $b \in \Rg$ in more detail.
By a basic result on reversible transition matrices, 
the right-eigenvectors of the $n \times n$ matrix $P(\gammac{a}{b})$
are orthogonal with respect to the inner product on $\mathbb{R}^n$ 
defined by
\[ \bigl\langle v, w \bigr\rangle =
\sum_{x=0}^{n-1} \pi_x v_x w_x \] 
where $\pi_x = 
\smbinom{n-x}{a} \smbinom{x+1}{a+b+1} / \smbinom{n}{2a+b+2}$ is the invariant distribution
in~\eqref{eq:gammaInvariant}. 
For $0 \le d < n$, let $\rP{d} \in \mathbb{R}^n$ be the 
right-eigenvector of $P(\gammac{a}{b})$ with
eigenvalue $(-1)^d\lambdac{a}{b}{d}$, normalized with respect to $\langle \ , \ \rangle$.
By Proposition~\ref{prop:practical}(iii), $\rP{d} \in \langle v(0), \ldots, v(d) \rangle$ for $0 \le d < n$.
Therefore $\rP{0}, \ldots, \rP{n-1}$ are the Gram--Schmidt orthonormalizations of $v(0)$, \ldots, $v(n-1)$,
with respect to $\langle \ , \ \rangle$. 
Since the eigenvalues
$(-1)^d\lambdac{a}{b}{d}$ of $P(\gammac{a}{b})$ are distinct, we also have
\[ \rP{e} \in \bigl\langle \prod_{d=0}^{e-1} \bigl( P(\gammac{a}{b}) - (-1)^d\lambdac{a}{b}{d} \bigr) v(e) \bigr\rangle \]
for $0 \le e < n$. 
Using either of these observations the right-eigenvectors can be computed
for any particular $n$. In particular, a routine calculation shows that $\rP{1}_x \propto (a+b+2)(n-1)
-(2a+b+3)x$.

The left-eigenvectors can be obtained using the following
lemma; we include a proof for completeness and to motivate Definition~\ref{defn:H} below.

\begin{lemma}\label{lemma:leftRight}
Let $P$ be the transition matrix of a finite reversible Markov chain with
strictly positive invariant distribution proportional to $\pi$. 
Let $v \in \mathbb{R}^n$ and let $\uw \in \mathbb{R}^n$ be defined by $\uw_x = \pi_x v_x$.
Then $v$ is a right-eigenvector for $P$
with eigenvalue $\lambda$ if and only if $\uw$ is a left-eigenvector for $P$ with eigenvalue $\lambda$.
\end{lemma}

\begin{proof}
We have
\[ (\uw P)_y = \sum_{x} \uw_x P_{xy} = \sum_{x} v_x \pi_x P_{xy} = \sum_{x} v_x\pi_y P_{yx} = \pi_y \sum_{x} v_x P_{yx} \]
where the penultimate equality uses the detailed balance equation
$\pi_x P_{xy} = \pi_y P_{yx}$. Now $v$ is a right-eigenvector for $P$ with
eigenvalue $\lambda$ if and only if the right-hand side is $\lambda \pi_y v_y$ for each $y$.
Since $\uw_y = \pi_y v_y$ and $\pi_y > 0$, this holds if and only if $\uw$ is a left-eigenvector for $P$
with eigenvalue $\lambda$.
\end{proof}

Remarkably, the left-eigenspace of $P(\gammac{a}{b})$ for the 
final eigenvalue, namely $\lambdac{a}{b}{n-1} = (-1)^{n-1} \smbinom{n}{a} / \smbinom{n}{a+b+1}$ 
does not depend on either $a$ or $b$; up to signs, it is
a row of Pascal's Triangle.

\begin{proposition}\label{prop:finalLeftEigenvector}
Let $u_x = (-1)^x \binom{n-1}{x}$. Then for all $a$, $b \in \Rg$ we have 
$u P(\gammac{a}{b}) = (-1)^{n-1} \smbinom{n}{a} / \smbinom{n}{a+b+1} u$.
\end{proposition}

\begin{proof}
By Lemma~\ref{lemma:chainWeightsTotal}(iv),
$P(\gammac{a}{b} )_{xy^\star} = H(\gammac{a}{b})_{xy}$. Hence 
$\bigl(uP(\gammac{a}{b})\bigr)_{y^\star} = (-1)^{n-1}\smbinom{n}{a} / \smbinom{n}{a+b+1} u_{y^\star}$ if and only if
\begin{equation}\label{eq:finalLeftEigenvectorCalculation} 
\sum_{x=y}^{n-1} \frac{ \smbinom{y+1}{a}\smbinom{x-y+1}{b}}{\smbinom{x+1}{a+b+1}} (-1)^x \binom{n-1}{x}
=  \frac{\smbinom{n}{a}}{\smbinom{n}{a+b+1}} (-1)^{y^\star} \binom{n-1}{y^\star}.\end{equation}
Observe that when $a$, $b \in \N$, we have
$\binom{n-1}{x} / \smbinom{x+1}{a+b+1} = \binom{n-1}{x} / \binom{x+a+b+1}{a+b+1} 
= \binom{n+a+b}{x+a+b+1} / \binom{n+a+b}{a+b+1} = \binom{n+a+b}{n-1-x} / \smbinom{n}{a+b+1}$.
Clearing denominators this becomes a polynomial identity, so it holds for all $a$, $b \in \Rg$.
Setting $w = n-1-y$, we use this to rewrite the left-hand side as
\[ \frac{\smbinom{y+1}{a}}{\smbinom{n}{a+b+1}} 
\sum_{w=0}^{n-1-y} \mbinom{n-w-y}{b} \binom{n+a+b}{w} (-1)^{n-1-w}. \]
By an instance of
Lemma~\ref{lemma:mbinom}(iv), this is $\smbinom{y+1}{a} (-1)^{n-1-y} \binom{n+a-1}{n-1-y}/
\smbinom{n}{a+b+1}$. Now observe that
\[ \mbinom{y+1}{a} \binom{n+a-1}{n-1-y} = \binom{y+a}{a} \binom{n-1+a}{y+a}
= \mbinom{n}{a} \binom{n-1}{y} \]
which since $\binom{n-1}{y} = \binom{n-1}{y^\star}$ gives the relevant factors in the right-hand side of~\eqref{eq:finalLeftEigenvectorCalculation}.
\end{proof}

The corollary for the right-eigenvectors of $P(\gammac{a}{b})$ is worth noting.

\begin{corollary}
Let $a$, $b \in \Rg$. Then 
\[ \rP{n-1}_x \propto (-1)^x \binom{n-1}{x} \bigl/ \mbinom{n-x}{a} \mbinom{x+1}{a+b+1}. \]
\end{corollary}

\begin{proof}
This is immediate from Lemma~\ref{lemma:leftRight} and the invariant distribution in~\eqref{eq:gammaInvariant}.
\end{proof}

In particular, if $a=0$ then from
\[ \binom{n+b}{x+b+1} \mbinom{x+1}{b+1} = \binom{n+b}{x+b+1} \binom{x+b+1}{b+1} 
= \binom{n+b}{b+1} \binom{n-1}{y} \]
we see that \smash{$\rP{n-1}_x \propto (-1)^x \binom{n+b}{x+b+1}$}. A slightly
more lengthy argument shows that if $a=1$ then
\smash{$\rP{n-1}_x \propto (-1)^{x}\binom{n+b+2}{x+b+2}$}.
However there do not appear to be such `division-free' formulae for higher $a$, or general $a \in \Rg$.
We return to these eigenvectors in \S\ref{subsec:discreteContinuous}.

\subsection{The involutive walk on subsets and proof of Corollary~\ref{cor:subsets}}
Let $0 < p < 1$. The $\gammac{-(1-p)}{-p}$-weighted involutive walk on $\{0,1\}$ has transition matrix
\[ Q = \left( \begin{matrix} \cdot & 1 \\ p & 1-p \end{matrix} \right). \]
By special cases of Theorem~\ref{thm:factorization} and
Theorem~\ref{thm:spectrum} the  invariant distribution is $\left( \mfrac{p}{1+p}, \mfrac{1}{1+p}
\right)$ and the eigenvalues are $1$ and $-p$. 
Observe that $Q$ is the transition matrix of the random walk in Corollary~\ref{cor:subsets} when $m=1$;
the empty set corresponds to
$0 \in \{0,1\}$ and $\{1\}$ to $1 \in \{0,1\}$.

Generally, if $P$ is the matrix of an involutive walk on $n$ and $P'$ is the matrix of an involutive walk
on $n'$ then, indexing the entries of $P \otimes P'$ by the ordinal product
$n \times n'$, we have $(P \otimes P')_{(x,x'),(z,z')} =
P_{(x,z)}P_{(x',z')}$, and so $P \otimes P'$ is the transition matrix of the random walk on $n \times n'$ in which,
starting from $(x, x')$, we choose $(y,y')$ with $y \in [x]$, $y' \in [x']$ with probability $P_{(x,y^\star)}P_{(x',{y'}^\star)}$
and then step to $(y^\star, {y'}^\star)$. It is easily seen that if the random walks on $n$ and on $n'$ 
are irreducible, reversible, recurrent and ergodic then so is the random walk on $n \times n'$.
In particular, the involutive walk on subsets
of $\{1,\ldots, m\}$ in Corollary~\ref{cor:subsets} has all these properties.
Its transition matrix $Q^{\otimes m}$ and so
its unique invariant distribution $\pi$ is \smash{$\left( \mfrac{p}{1+p}, \mfrac{1}{1+p}
\right)^{\raisebox{-4pt}{$\scriptstyle \otimes m$}}$}; by the identification at the end of the previous paragraph we have $\pi_X =
\bigl( \mfrac{p}{1+p} \bigr)^{m-|X|} \bigl( \mfrac{1}{1+p} \bigr)^{|X|}$ for each $X \subseteq \{1,\ldots, m\}$.
Therefore $\pi_X = p^{m-|X|}/(1+p)^m$ as claimed. 
It is easily seen, either directly or by Proposition~\ref{prop:finalLeftEigenvector},
that the left-eigenvector of~$Q$ with eigenvalue $-p$ is $(1,-1)$. Therefore 
the $(-p)^e$-eigenspace is spanned by all $m$-fold tensor products of $(p,1)$ and $(1,-1)$ and has
dimension $\binom{m}{e}$.
This completes the proof of Corollary~\ref{cor:subsets}.

\section{When is a binomial transform stochastic?}
\label{sec:stochastic}
Given a sequence $\lambda_0$, \ldots, $\lambda_{n-1}$ of real numbers, recall from~\eqref{eq:HL}
that $\DSL$ denotes the binomial transform 
$B(n)\Diag(\lambda_0,\ldots,\lambda_{n-1})B(n)^{-1}$. By Theorem~\ref{thm:GB},
an $n \times n$ matrix has global anti-diagonal eigenbasis action if and only if it is equal to 
some $\DSL$, and, by Theorem~\ref{thm:OchiaiEtAl}, this property also characterises generic
matrices with the global anti-diagonal eigenvalue property.
We have defined
$\PL = \DSL J(n)$ to be the anti-triangular matrix obtained from~$H^\lambda$.
In this section we determine when $\PL$ is stochastic, and, supposing that $\PL$ is stochastic,
when the involutive walk it defines is ergodic.

The introduction discusses the connection with the results
on half-infinite matrices in \cite{OchiaiEtAl}. Our approach, using the following two lemmas, is more elementary.

\begin{lemma}\label{lemma:binomialTransformRecurrence}
Let $\lambda_0, \ldots, \lambda_{n-1} \in \R$ and let $x, y \in \n$ with $x < n-1$.
If $x \ge y$ then
\[ \DSL_{x+1,y}\binom{x+1}{y}^{-1}\!\! = \DSL_{xy}\binom{x}{y}^{-1}\!\! - \DSL_{x+1,y+1}\binom{x+1}{y+1}^{-1}\!\!.\]
\end{lemma}

\begin{proof}
Let $k = x-y \in \N_0$. By Lemma~\ref{lemma:practical}(iv),
the coefficient of $\lambda_{y+e}$ in $\DSL_{xy}\binom{x}{y}^{-1}$, $\DSL_{x+1,y+1}\binom{x+1}{k}^{-1}$ 
and $\DSL_{x+1,y}\binom{x+1}{y}^{-1}$ is
$(-1)^e \binom{k}{e}$, $(-1)^{e-1} \binom{k}{e-1}$ and $(-1)^e \binom{k+1}{e}$ respectively,
for $e \in \{1,\ldots,n-1-y\}$; for $e=0$ the coefficient is $1$, $0$ and $1$, respectively.
The lemma therefore follows from the basic identity $\binom{k}{e} + \binom{k}{e-1} = \binom{k+1}{e}$
for $e\in\{1,\ldots, k\}$.
\end{proof}

\begin{lemma}\label{lemma:nonNegative}
Let $\lambda_0,  \ldots, \lambda_{n-1} \in \R$. Then
\begin{thmlist}
\item $\DSL$ is non-negative if and only if its bottom row is non-negative;
\item if $\DSL$ is non-negative and $\DSL_{xy} = 0$ then $\DSL_{x'y'} = 0$ for all $x'$, $y' \in n$ with
$x'-x \ge y'-y$.
\item $\DSL_{x+1,x} = \binom{x+1}{x}(\lambda_x - \lambda_{x+1})$ for all $x \in \{0,1,\ldots, n-2\}$;
\item if $\DSL$ is non-negative and $\lambda_t = \lambda_{t+1}$ then $\lambda_t = \ldots = \lambda_{n-1}$.
\end{thmlist}
\end{lemma}

\begin{proof}
As a technical tool,
we define $F^\lambda$ to be the lower-triangular 
$n \times n$ matrix with entries $F^\lambda_{xy} = \DSL_{xy}\binom{x}{y}^{-1}$ for $x \ge y$ and $F^\lambda_{xy} = 0$ for $x < y$.
By Lemma~\ref{lemma:binomialTransformRecurrence}, 
the entries of $F$ satisfy the recurrence
\[ F^\lambda_{x+1,y} = F^\lambda_{xy} - F^\lambda_{x+1,y+1}. \]
Hence if
$F^\lambda_{xy} < 0$ then either $F^\lambda_{x+1,y} < 0$ or $F^\lambda_{x+1,y+1} < 0$. Thus
a negative entry in row $x$ of $F^\lambda$ implies a negative entry in row $x+1$ of $F^\lambda$.
It follows that~$F^\lambda$ is non-negative if and only if its bottom row is non-negative, and the same holds for $\DSL$.
This proves~(i). Suppose that $\DSL$ is non-negative and that $\DSL_{xy} = 0$. Then $F^\lambda_{xy} = 0$
and by the recurrence relation, $F^\lambda_{x+1,y} = F^\lambda_{x+1,y+1} = 0$. It follows
that $H^\lambda_{x'y'} = 0$ for all $x' \ge x$ and $y' \ge y$ with $x' - x \ge y' - y$, proving (ii). 
Part (iii) also follows from the recurrence, using that $F^\lambda_{xx} = \lambda_x$.
If $\lambda_t = \lambda_{t+1}$ then, by (iii),~$F^\lambda_{t+1,t} = 0$, and so by (ii),
$F^\lambda_{x+1,x} = 0$ for all $x \ge t$. Hence, by another use of (iii), 
$\lambda_x = \lambda_{x+1}$ for all $x \ge t$, as required.
\end{proof}

\setcounter{theorem}{5}
\setcounter{section}{1}
\begin{corollary}\corollaryBinomialTransform
\end{corollary}

\setcounter{theorem}{2}
\setcounter{section}{6}

\begin{proof}
By Lemma~\ref{lemma:practical}(iv), $H^\lambda_{n-1,y} = \sum_{e=0}^{y} (-1)^e \binom{n-1-y}{e} \lambda_{z^\star + e}$.
Since this is the entry in column $y^\star$ of the bottom row of $\PL$, it follows from
Lemma~\ref{lemma:nonNegative}(i) that $\PL$ is non-negative if and only if
$\sum_{e=0}^{z^\star} (-1)^e \binom{z}{e} \lambda_{z^\star + e} \ge 0$ for all \hbox{$z \in n$}.
Moreover, by Lemma~\ref{lemma:nonNegative}(iv), if strict inequality
holds for all $z \in n$ then $\lambda_0 > \lambda_1 > \ldots > \lambda_{n-1}$ and, by
Lemma~\ref{lemma:nonNegative}(ii),
$\PL$ has strictly positive entries in all positions below its
anti-diagonal.
Recall that $v(0) \in \mathbb{R}^n$ is the all-ones vector.
By Lemma~\ref{lemma:practical}(ii), $\DSL v(0) = \lambda_0 v(0)$. Therefore~$\DSL$, and hence $\PL$, 
has all row sums equal to $1$ if and only if $\lambda_0 = 1$. 
\end{proof}

Either by taking a linear combination of the inequalities in this corollary, or more directly 
from Lemma~\ref{lemma:nonNegative}(iii), it follows that
if $\PL$ is a transition matrix then $\lambda_0 \ge \lambda_1 \ge \ldots \ge \lambda_{n-1}$,
and so the below-diagonal entries of $\binom{x+1}{x}(\lambda_x-\lambda_{x+1})$
are non-negative.
Similarly, from the entries below these, we get the convexity property
$(\lambda_{x-1} + \lambda_{x+1})/2 \ge \lambda_{x}$ for all $x \in n$ such that $1 \le x < n-1$.
As mentioned in the introduction, the full condition in Corollary~\ref{cor:binomialTransformIsInvolutiveWalk} 
implies that the eigenvalues are completely monotonic.
This is of course the case for
the eigenvalues $\lambdac{a}{b}{d}$ of the matrices $P(\gammac{a}{b})$, $P(\gammai{c})$ and $P(\deltac{a'}{b'})$.
Since these transition matrices are of irreducible ergodic walks, 
it is natural to ask, more generally, when is the involutive walk defined by~$P^\lambda$ 
irreducible and ergodic? This has a neat answer, although the proof, using the following  proposition,
is unavoidably somewhat technical. 

\begin{proposition}\label{prop:PLergodic}
Suppose that $P^\lambda$ is the transition matrix of an involutive walk
in which $0$ is accessible from all states.
Let $s \in \N$ be maximal such that $\lambda_{n-s} = \ldots = \lambda_{n-1}$. Then $s \le n/2$ and
$1 > \lambda_1 > \ldots > \lambda_{n-s-1} > \lambda_{n-s} = \ldots = \lambda_{n-1} > 0$ 
and the walk is
irreducible, recurrent and ergodic with a unique strictly positive invariant distribution.
\end{proposition}

\begin{proof} 
By Corollary~\ref{cor:binomialTransform}, $\lambda_0 = 1$.
Since $0$ can be reached in one step only from $n-1$, and $\PL_{n-1,0} = \DSL_{n-1,n-1} = \lambda_{n-1}$, 
we have $\lambda_{n-1} > 0$.
Hence, by Lemma~\ref{lemma:nonNegative}(iii) and (iv),
$\lambda_0 > \lambda_1 > \ldots > \lambda_{n-s-1} > \lambda_{n-s} = \ldots
= \lambda_{n-1} > 0$.
By Lemma~\ref{lemma:nonNegative}(iii), $\DSL_{n-s+1,n-s} = \binom{n-s+1}{n-s} (\lambda_{n-s} - \lambda_{n-s+1}) = 0$.
Hence, by Lemma~\ref{lemma:nonNegative}(ii), $\DSL_{xy} = 0$ for all $x \ge n-s+1$, $y \ge n-s$ with $x - (n-s+1) \ge y - (n-s)$. Thus if $y > n-1-s$ then column $y$ of $\DSL$ is zero except for its diagonal entry $\lambda_y$.
Since $s^\star = n-1-s$, it follows that if $y > s^\star$ then
column $y$ of~$\DSL$ has a unique non-zero entry of $\lambda_{y}$ in row~$y$.

Suppose, for a contradiction, that $s > n/2$. If $n$ is even, with $n=2m$ then  $s^\star < m-1$,
and the unique non-zero entries of $H^\lambda$ in columns $m-1$ and $m$ are on the diagonal.
Thus from $m-1$ the involutive walk must step to $(m-1)^\star = m$, and from $m$ it must step to $m^\star = m-1$,
contradicting that~$0$ is accessible from any state. Similarly, if $n$ is odd, with $n=2m+1$ then 
$s^\star < m$ and from $m$ the involutive walk must step to $m^\star = m$. Therefore~$s \le n/2$.

Again by Lemma~\ref{lemma:nonNegative}(iii), $\DSL_{s^\star+1,s^\star} = \DSL_{n-s,n-s-1} = 
(n-s)(\lambda_{n-s-1}
-\lambda_{n-s}) > 0$.
By Lemma~\ref{lemma:binomialTransformRecurrence},
\[ H^\lambda_{x,s^\star} \binom{x}{s^\star}^{-1}\!\!\! = H^\lambda_{x-1,s^\star} \binom{x-1}{s^\star}^{-1} \!\!\!
- H^\lambda_{x,s^\star+1} \binom{x}{s^\star +1}^{-1}\!\!
= H^\lambda_{x-1,s^\star} \binom{x-1}{s^\star}^{-1}\!\! \]
for $x > s^\star$.
It follows inductively that all the entries on or below the diagonal in column $s^\star$ of $H^\lambda$ are non-zero.
To summarise so far, we have shown that
the down-steps $x \mapsto x$ for $x \in n$, $x \mapsto x-1$ for $x \in \{1,\ldots, s^\star+1\}$ and $n-1 \mapsto s^\star$ have non-zero probability. 
Moreover, since $s^\star + 1 - s = (n-1-s) + 1 - s = n - 2s \ge 0$, we have $s \le s^\star + 1$.
It is now routine to check that there is a cycle in the involutive walk
\[ n-1 \mapsto s \dmapsto s^\star+1 \mapsto s-1 \dmapsto s^\star + 2 \mapsto \ldots \dmapsto n-2 \mapsto 1 \dmapsto n-1 \]
where bold arrows indicate the steps whose down-step is $x \mapsto x-1$. 
Observe that $t-1$ appears two steps before $t$ for each $t \in \{s^\star +2, \ldots, n-1\}$. Therefore
$t-1$ is accessible from $t$ for each such $t$. Since the step $x \mapsto x^\star$ has non-zero probability for all $x$,
it follows that $(t-1)^\star$ is accessible from $t$ for each $t \in \{s^\star +2, \ldots, n-1\}$. Since
the step $x \mapsto (x-1)^\star$ has positive probability for each $x \in \{1,\ldots,s^\star + 1\}$
the hypothesis of Lemma~\ref{lemma:longWalk} holds. Therefore
the involutive walk is irreducible.
Moreover, as in the proof of Theorem~\ref{thm:ergodic},
 the step $m \mapsto m$ where $m = \lfloor n/2 \rfloor$
has non-zero probability. Hence the chain is aperiodic and ergodic with a unique strictly positive
invariant distribution.
\end{proof}

\begin{corollary}\label{cor:PLergodic}
Suppose that $P^\lambda$ is the transition matrix of an involutive walk.
Then $0$ is accessible from all states if and only if 
the walk is
irreducible, recurrent and ergodic with a unique strictly positive invariant distribution.
\end{corollary}

\begin{proof}
The `only if' direction is proved by Proposition~\ref{prop:PLergodic}. Conversely, if the walk is irreducible
then $0$ is accessible.
\end{proof}

\section{Binomial transforms and reversibility}
\label{sec:reversibility}

Recall that $P^\lambda = H^\lambda J(n)$ where $H^\lambda$ is the binomial transform defined in~\eqref{eq:HL}.
By Theorem~\ref{thm:spectrum} and the remark following~\eqref{eq:HL},
the $\gammac{a}{b}$-, $\deltac{a'}{b'}$- and $\gammai{c}$-weighted involutive walks on $n$ 
are reversible for any $n$ in their domain and their transition matrices are of the form $P^\lambda$.
This proves the `if' direction of Theorem~\ref{thm:reversibleImpliesGamma}.
The converse is the surprising part and is proved in this section.

Our starting hypothesis is that $n \ge 3$ and $P^\lambda$ is a globally reversible 
(see Definition~\ref{defn:globallyReversible}) $n \times n$ transition matrix of an involutive walk
in which $0$ is accessible from every state. The corresponding down-step matrix $\DSL$ is
a binomial transform, as in~\eqref{eq:HL}, with eigenvalues $\lambda_0, \ldots, \lambda_{n-1}$.
By Corollary~\ref{cor:binomialTransform}, $\lambda_0 = 1$.
Throughout, we let $\mu = \lambda_1$ and $\nu = \lambda_2$.
If $\mu = 1$ then by Lemma~\ref{lemma:nonNegative}(iii), $\DSL_{10} = 0$, and so by Lemma~\ref{lemma:nonNegative}(iii) and (iv) it 
follows that $P^\lambda = J(n)$. This contradicts the assumption
that $0$ is accessible from every state. Hence 
\[ \tag{$\dagger$} 1 > \mu \ge \nu \ge \lambda_3 \ge \ldots \ge \lambda_{n-1} \ge 0. \]

Since the down-step matrices
$H(\gammac{a}{b})$, $H(\deltac{a'}{b'})$ and $H(\gammai{c})$ are binomial transforms, 
We have the following key principle: if $\gamma$ is the relevant weight $\gammac{a}{b}$, $\gammai{c}$ or
$\deltac{a'}{b'}$ and $\lambda_d(\gamma) = H(\gamma)_{dd}$ for $d \in \{0,\ldots,n-1\}$ then
\[ 
P^\lambda = P(\gamma) \iff \text{\parbox{3in}{\raggedright for all $d \in \ne$, the eigenvalue $\lambda_d$ of $H^\lambda$
is equal to the eigenvalue $\lambda_d(\gamma)$ of $H(\gamma)$.}}
\tag{$\star$} \]
It is therefore important that the eigenvalues $\mu$ and $\nu$ determine the parameters $a$ and $b$.
By~\eqref{eq:evaluesIntroduction},~\eqkey~requires
\begin{equation}\label{eq:munu} \frac{a+1}{a+b+2} = \mu, \quad 
\frac{(a+1)(a+2)}{(a+b+2)(a+b+3)} = \nu. \end{equation}
Provided $\nu \not= \mu^2$, these equations have the unique solution specified by the functions
\begin{equation}\label{eq:abEqns}
a(\mu, \nu) = \frac{\mu(\mu-\nu)}{\nu-\mu^2} - 1, \quad b(\mu, \nu) = \frac{(1-\mu)(\mu-\nu)}{\nu-\mu^2} - 1.
\end{equation}

This gives some motivation for the overall strategy of the proof, which is by induction on $n$, 
using the global reversibility hypothesis to show that
the three greatest eigenvalues $1, \mu, \nu$ of $H^\lambda$ 
determine all $n$ eigenvalues $\lambda_0, \ldots, \lambda_{n-1}$, and hence that the right-hand side in~\eqkey\ holds.

We structure the argument as follows: \S\ref{subsec:rKolmogorov} gives a version of Kolmogorov's Criterion
adapted to involutive walks; \S\ref{subsec:rbase} proves
Theorem~\ref{thm:reversibleImpliesGamma} when $n=3$, and motivates \S\ref{subsec:rinductive} where we formulate the 
stronger Proposition~\ref{prop:reversibleImpliesGamma}; finally in \S\ref{subsec:rproof} we prove this proposition by induction
on $n$, using Kolmogorov's Criterion on a particular $4$-cycle, dealing with the case $\nu = 2\mu-1$, when this cycle has zero probability, separately.

\subsection{Kolmogorov's Criterion}\label{subsec:rKolmogorov}
One form of Kolmogorov's Criterion (see for instance
\cite[\S 1.5]{KellyReversibilityStochasticNetworks}) is that a random walk with a strictly positive invariant distribution
is reversible if and only if the product of transition probabilities around any cycle does not
depend on the direction of travel. The special case for involutive walks is stated in (ii) below.

\begin{lemma}\label{lemma:reversibilityCriteria}
Let $P$ be the transition matrix of an involutive walk with strictly positive invariant distribution.
\begin{thmlist}
\item If the walk is reversible then for all $x$, $z \in n$, we have $P_{xz} > 0$ if and only if $P_{zx} > 0$.
\item The walk is reversible if and only~if
\[ P_{x_0x_1}P_{x_1x_2}\ldots P_{x_{\ell-1}x_0} = 
P_{x_0x_{\ell-1}} \ldots P_{x_2x_1} P_{x_1x_0} \]
for all distinct  $x_0, x_1, \ldots, x_{\ell-1} \in n$ with $\ell \ge 3$,
such that $x_i + x_{i+1} \ge n-1$ for all $i \in \{0,1,\ldots, \ell-2\}$, and $x_{\ell-1} + x_0 \ge n-1$.
\end{thmlist}
\end{lemma}

\begin{proof}
From the detailed balance equations $\pi_x P_{xz} = \pi_z P_{zx}$ we
get $P_{xz} > 0$ if and only if  $P_{zx} > 0$, proving (i).
Since $P_{xz} = 0$ unless $x + z \ge n-1$, it suffices to verify Kolmogorov's Criterion for cycles
of the form $x_0 \mapsto x_1 \mapsto x_2 \mapsto \ldots \mapsto x_{\ell-1} \mapsto x_0$ where $x_i + x_{i+1} 
\ge n-1$ for all~$i \in \{0,1,\ldots, \ell-2\}$ and $x_{\ell-1} + x_0 \ge n-1$. This is trivial for $2$-cycles, so we may suppose that $\ell \ge 3$, giving~(ii).
\end{proof}

\subsection{Base case: $n=3$}\label{subsec:rbase}
If $\nu > \mu^2$ then defining $a$ and $b$ by~\eqref{eq:abEqns} we have $a$, $b \in \Rg$ and 
the eigenvalues of $\DSL$ and $H(\gammac{a}{b})$ agree, so by~\eqkey, $\DSL = H(\gammac{a}{b})$.
If $\nu = \mu^2$ then $\mu < 1$ by ($\dagger$), and so by~\eqkey, $H^\lambda = \DS(\gammai{c})$ where $1/(c+1) = \mu$,
or equivalently, $c = \frac{1-\mu}{\mu}$.
In the remaining case $\nu < \mu^2$ and again
defining $a$ and $b$ by~\eqref{eq:abEqns} we have $a$, $b < -1$.
By either Lemma~\ref{lemma:practical}(iv) or Lemma~\ref{lemma:binomialTransformRecurrence},
\[ H^\lambda = \left(\begin{matrix} 1 & \cdot & \cdot \\ 1-\mu & \mu & \cdot \\ 1-2\mu + \nu & 2(\mu-\nu) & \nu  
\end{matrix} \right) \]
and so $\nu \ge 2\mu-1$. Moreover, since $0$ is accessible from all states, $\nu > 0$.
Setting $\nu =2\mu-1+\tau(1-\mu)$ we have
$\mu-\nu = (1-\mu)(1-\tau)$ and $\mu-\nu^2 = (1-\mu)(1-\mu - \tau)$, hence $(1-\mu)(\mu-\nu)/(\nu-\mu^2) = 
-(1-\mu)(1-\tau)/(1-\mu-\tau)$ where $1-\mu-\tau > 0$ and 
\[ -b\bigl( \mu,2\mu-1+\tau(1-\mu)\bigr) = 2 + \frac{\mu \tau}{1-\mu - \tau}, \]
for $0 \le \tau < 1-\mu$. Similar calculations show that
\[ -a\bigl( \mu, 2\mu-1+\tau(1-\mu) \bigr) = \frac{1 - \tau- \mu\tau}{1 - \tau - \mu}.\]
Comparing the numerator and denominator immediately above we see that $-a\bigl( \mu, 2\mu-1+\tau(1-\mu) \bigr)$
is an increasing function of $\tau$ with minimum $\mfrac{1}{1-\mu}$ when $\tau = 0$. Since
$\nu > 0$ we have $\mu > \mfrac{1}{2}$ and so $a' > 2$.
Moreover, when $\tau = 0$, we have $b' = 2$ as in the final case of Theorem~\ref{thm:reversibleImpliesGamma}.
Otherwise $\tau > 0$ and $b' > 2$, again as required.
This completes the proof of the `only if' direction when $n=3$.

\subsection{Inductive claim}\label{subsec:rinductive}
To enable an inductive proof, we must prove a stronger version of Theorem~\ref{thm:reversibleImpliesGamma}
that takes into account 
that the exceptional case $\nu = 2\mu-1$ and $a'=\mfrac{1}{1-\mu}$
is the first in an infinite family. 

To motivate the statement of Proposition~\ref{prop:reversibleImpliesGamma}, observe that if $m \in \N$ then 
\smash{$\deltac{a'}{m}_{[y,x]} = \binom{a'-1}{y}\binom{m-1}{x-y}$} vanishes if and only if $x-y \ge m$.
Thus \emph{provided $n \ge m$}, the transition matrix $P(\deltac{a'}{m})$ has 
exactly $m$ non-zero anti-diagonal bands. 
This was seen, when $m=2$, in one of the examples before Theorem~\ref{thm:spectrum}, and also in the base case just proved,
for which $-b(\mu, \nu) = b(\mu, 2\mu-1) = 2$ and $\mu > \mfrac{1}{2}$. 
Therefore we may assume $\mu > 1/2$ when
finding these exceptional cases. Let
$b'_\mu(\nu) = -b(\mu, \nu)$ and note that 
\[ \frac{\mathrm{d}b'_\mu(\nu)}{\mathrm{d}\nu} = \frac{\mu(1-\mu)^2}{(\mu^2-\nu)^2}.\]
Thus $b'_\mu$ is increasing on the interval $[0,\mu^2)$, 
with $b'_\mu(0)= \mfrac{1}{\mu}$ and $b'_\mu(\nu) \rightarrow \infty$ as $\nu \rightarrow \mu^2$.
Since $\mu > \mfrac{1}{2}$, for each $m \in \N$ with $m \ge 2$ there exists a unique~$\nu$ such
that $b'_\mu(\nu) = m$. We denote this value by $\nu_m(\mu)$. Let $a'_m(\nu) = -a(\mu, \nu_m(\mu))$. By
inverting $b'_\mu$ and then substituting one gets the explicit formulae
\begin{equation}\label{eq:mMu} \nu_m(\mu) = \frac{\mu(m\mu-1)}{m-2+\mu}, \quad 
a'_m(\mu) = \frac{(m-2)\mu +1}{1-\mu}.\end{equation}
\noindent In particular,
 $\nu_2(\mu) = 2\mu-1$ and $a'_2(\mu) = \mfrac{1}{1-\mu}$, as seen in the case $n=3$. 
 Since $b'_\mu$ is an increasing function,
so is the sequence $\nu_m(\mu)$. 
Moreover, since $b'_\mu(\nu) \rightarrow \infty$
as $\nu \rightarrow \mu^2$, the sequence $\nu_m(\mu)$ converges to $\mu^2$ as $m \rightarrow \infty$.
This can be seen in the graph in Figure~1 showing $a(\mfrac{2}{3},\nu)$ and $b(\mfrac{2}{3}, \nu)$ 
and the exceptional values $\nu_m(\mfrac{2}{3})$ for $m \in \{2,3,4,5,6\}$.

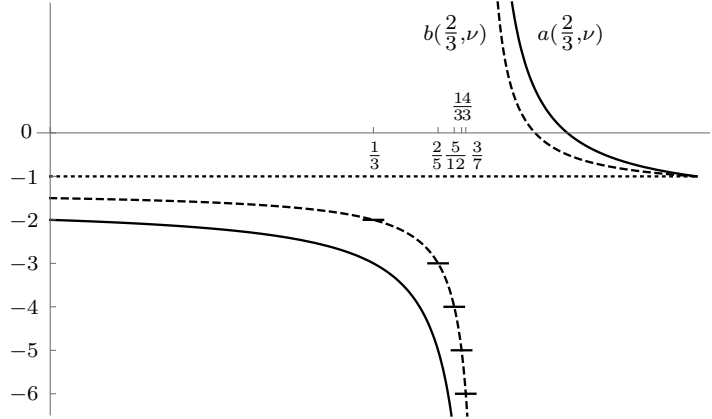
\begin{figure}
\begin{center}\setlength{\unitlength}{1cm}
\begin{picture}(5,5.5)
\put(-1,0) {\includegraphics[width=3.5in]{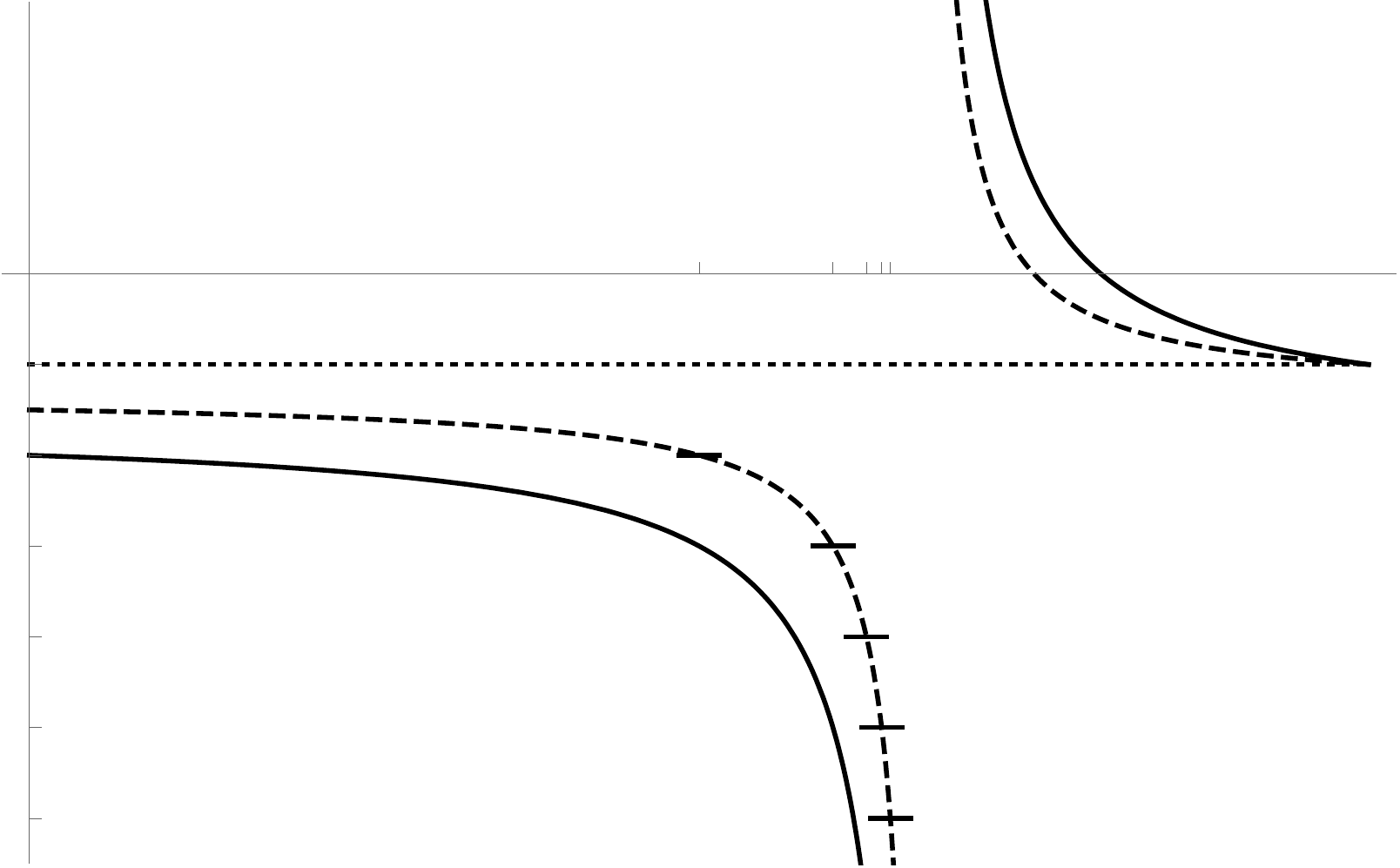}}
\put(5.6,5) {$\scriptstyle a(\mfrac{2}{3}, \nu)$}
\put(4.1,5) {$\scriptstyle b(\mfrac{2}{3}, \nu)$}
\put(-1.2,3.7) {$\scriptstyle 0$}
\put(-1.35,3.1) {$\scriptstyle -1$}
\put(-1.35,2.525) {$\scriptstyle -2$}
\put(-1.35,1.95) {$\scriptstyle -3$}
\put(-1.35,1.375) {$\scriptstyle -4$}
\put(-1.35,0.8) {$\scriptstyle -5$}
\put(-1.35,0.225) {$\scriptstyle -6$}
\put(3.35,3.4) {$\sfrac{1}{3}$}
\put(4.17,3.4) {$\sfrac{2}{5}$}
\put(4.35,3.4) {$\sfrac{5}{12}$}
\put(4.45,4.05) {$\sfrac{1\hskip-0.5pt4}{3\hskip-0.5pt3}$}
\put(4.7,3.4) {$\sfrac{3}{7}$}
\end{picture}
\end{center}
\caption{The functions $a(\mfrac{2}{3},\nu)$ and $b(\mfrac{2}{3}, \nu)$ for
$0 \le \nu \le 2/3$
and the exceptional values $\nu_m(\mfrac{2}{3})$ for $m \in \{2,3,4,5,6\}$, namely
$\mfrac{1}{3}, \mfrac{2}{5}, \mfrac{4}{12}, \mfrac{14}{33}, \mfrac{3}{7}$.
Each has a vertical
asymptote at $(\mfrac{2}{3})^2$.} 
\end{figure}

Instead fixing $m$, 
the relevant values of~$n$ are  given the lemma below, which shows that the chosen
domains of $\deltac{a'}{b'}$ for $b' \not\in \N$ and $\deltac{a'}{b'}$ for $b' \in \N$ specified 
after~\eqref{eq:delta}
are as large as possible.

\begin{samepage}
\begin{lemma}\label{lemma:deltaDomainSimplified}
Let $a'$, $b' \in \R^{>1}$.
\begin{thmlist}
\item If $b' \not\in \N$ then the 
function $\deltac{a'}{m}$ on intervals of $n$ is a weight
such that $0$ is accessible from every state in the $\deltac{a'}{b'}$-weighted involutive walk
if and only if $n \le \min( \lceil a' \rceil, \lceil b' \rceil)$ and no larger domain is possible.

\item Suppose that $b' = m \in \N$. The function $\deltac{a'}{m}$ on intervals of $n$ is a weight
such that $0$ is accessible from every state in the $\deltac{a'}{m}$-weighted involutive walk
if and only if $n \le \lceil a' \rceil$. Moreover, $m \ge \lfloor \mfrac{1-\mu}{\mu} (n-2) \rfloor + 2$,
where $\mu$ is the usual eigenvalue.
\end{thmlist}
\end{lemma}
\end{samepage}

\begin{proof}
By~\eqref{eq:delta}, $\deltac{a'}{m}_{[y,x]} = \binom{a'-1}{y}\binom{b'-1}{x-y}$.
Thus if $n > \lceil a' \rceil$ 
then $\deltac{a'}{b'}_{[n-1, n-1]}$ is either negative, contrary to the definition
of weight, or zero, implying that $(n-1)^\star = 0$ is inaccessible from $n-1$, and so
inaccessible from any state.
Similarly, if $b' \not\in \N$ then taking $n > \lceil b' \rceil$ we find
that $\deltac{a'}{b'}_{[0,n-1]}$ is negative. Thus the largest possible domain is $\min(\lceil a' \rceil,
\lceil b'\rceil)$, as claimed.
When $b' \in \N$, as in (ii), there is no constraint from $\binom{b'-1}{x-y}$,
and so the largest domain is instead~$\lceil a' \rceil$.
To complete the proof, we use the 
equation for $a'_m(\mu)$ in~\eqref{eq:mMu} to get
$n \le \lceil \frac{(m-2)\mu +1}{1-\mu} \rceil$. Hence \smash{$n < \frac{(m-2)\mu+1}{1-\mu} + 1$},
and so $(1-\mu)n < (m-2)\mu + 2 - \mu$ which rearranges to
 \smash{$m > \mfrac{1-\mu}{\mu}n - \mfrac{2}{\mu} + 3 = \mfrac{1-\mu}{\mu} (n-2) +1$}, or equivalently,
$m \ge \lfloor  \frac{1-\mu}{\mu}(n-2) \rfloor + 2$.
\end{proof}

We can now state our precise result, knowing that all the transition matrices specified are well-defined and
give involutive walks in which $0$ is accessible from every state; the condition $\mu > \mfrac{1}{2}$ in (iv)
was seen when $n=3$ in the proof of the base case.

\begin{proposition}\label{prop:reversibleImpliesGamma}
Let $\PL$ be the transition matrix of a reversible Markov chain on $n$ where $n\ge 3$ in which 
$0$ is accessible from every state.
Then $\lambda_0= 1$. Let $\mu = \lambda_1$, $\nu = \lambda_2$. 
One of the following cases applies.
\begin{thmlist}
\item If $\nu > \mu^2$ then $\mu > \nu$ and
the unique $a$, $b \in \Rg$ such that $\PL = P(\gammac{a}{b})$ are $a(\mu, \nu)$ and $b(\mu, \nu)$.
\item If $\nu = \mu^2$ then $\mu < 1$ and $\PL= P(\gammai{c})$ where $c = (1-\mu)/\mu$.
\item 
If $\nu_{n-1}(\mu) < \nu < \mu^2$ then the unique $a'$, $b' \in \R^{>1}$
such that $\PL = P(\deltac{a'}{b'})$ are $-a(\mu, \nu)$ and $-b(\mu, \nu)$ and $n \le \min(\lceil a' \rceil, \lceil b' \rceil)$.
\item Otherwise, $\mu > \mfrac{1}{2}$, $\nu = \nu_m(\mu)$ for a unique $m \in n$ such that 
$m \ge \lfloor \mfrac{1-\mu}{\mu} (n-2) \rfloor + 2$
and the unique $a' \in \R^{> 1}$ such
that $\PL = P(\deltac{a'}{m})$ is $a'_m(\nu)$. Moreover $n \le \lceil a'_m(\nu) \rceil$ and 
$P(\deltac{a'}{m})$ has exactly $m$ anti-diagonal bands where $m < n$.
\end{thmlist}
\end{proposition}

The example below shows the interesting behaviour in (iii) and~(iv).

\begin{example}\label{ex:reversibleSpectrum}
Take $n=10$. As in Figure~1, we take $\mu = \mfrac{2}{3}$. A  special feature of this case
is that
$a'_m(\mu) = 2m-1$ is integral.
By (iii) and~(iv) in Proposition~\ref{prop:reversibleImpliesGamma}, either 
$\nu >  \nu_9(\mfrac{2}{3}) 
= \mfrac{10}{23}$,
 or $\nu$, $a'$, $b'$ 
are in the table below.

\medskip
\begin{center}
\begin{tabular}{cccccc} \toprule 
$\nu$  & $\mfrac{10}{23}$ & $\mfrac{13}{30}$ & $\mfrac{22}{51}$ & $\mfrac{3}{7}$ \\ 
$a' = a'(\mfrac{2}{3}, \nu)$   & 17 & 15 & 13 & 11 \\
$b' = m $  & 9  & 8 & 7 & 6  \\ \bottomrule \end{tabular}
\end{center}
\smallskip
\noindent By (iv), or equivalently Lemma~\ref{lemma:deltaDomainSimplified}, we require $m 
\ge \lfloor \frac{1-\mu}{\mu} (n-2) \rfloor + 2 = \mfrac{n}{2} + 1 = 6$, and so the table stops with $m=6$.
The next entry, were it permitted, would be $\nu = \mfrac{14}{45}$, $a'=9$, $b'=5$, but $9 \not> 10-1$, 
and correspondingly $10$ is not in the domain of the weight $\deltac{9}{5}$. 
Thus when $n=10$, the reversible random walks 
with second largest eigenvalue~$\mfrac{2}{3}$ classified in Proposition~\ref{prop:reversibleImpliesGamma} 
split into a continuously varying family when $\mfrac{4}{9} < \nu < \mfrac{2}{3}$, 
a one-off case when $\nu = \mfrac{4}{9}$, a  continuously varying family when $\mfrac{10}{23} < \nu < \mfrac{4}{9}$,
and final special cases for $\nu = \mfrac{10}{23}, \mfrac{13}{30}, \mfrac{22}{51}, \mfrac{3}{7}$,
in which the transition matrix has respectively $9$, $8$, $7$ and~$6$ non-zero anti-diagonal bands.
To illustrative the inductive step in the proof of Proposition~\ref{prop:reversibleImpliesGamma}
suppose we increase $n$ to~$11$. Then we 
require either $\nu > \nu_{10}(\mfrac{2}{3}) = \mfrac{17}{39}$ or $\nu = \mfrac{17}{39}$, $a' = 19$, $m = 10$,
losing the cases where $\mfrac{10}{23} < \nu < \mfrac{17}{39}$ for $n=10$.
\end{example}

\enlargethispage{6pt}

\subsection{Proof of Proposition~\ref{prop:reversibleImpliesGamma}}\label{subsec:rproof}

\subsubsection*{Case $\nu = 2\mu-1$} In
this case $\PL$ has just two non-zero anti-diagonals. We exploit this in the proof.

\begin{lemma}\label{lemma:nuLeast}
Let $n \ge 3$. If $\nu = 2\mu-1$ then $\DSL= H(\deltac{\sfrac{1}{1-\mu}}{2})$ and \emph{(iv)} in Proposition~\ref{prop:reversibleImpliesGamma} holds with $m=2$.
\end{lemma}

\begin{proof}
It follows either directly from~\eqref{eq:delta} or using~Lemma~\ref{lemma:binomialTransformRecurrence}, that
the matrix $H(\deltac{\sfrac{1}{1-\mu}}{2})$ has entries
\begin{equation}\label{eq:DSdelta2} \DS(\deltac{\sfrac{1}{1-\mu}}{2})_{xy} = \begin{cases} x\mu - (x-1) & \text{if $x=y$} \\ 
x(1-\mu) & \text{if $x= y+1$} \\
0 & \text{otherwise.} \end{cases} \end{equation}
When $n=3$, this has the same eigenvalues, namely $1$, $\mu$, $2\mu -1$ as $\DSL$, and so the
matrices are equal by~\eqkey, and by the base case, $\mu > \mfrac{1}{2}$. Suppose inductively that~\eqkey\ holds for $n-1$. Thus the $n \times n$ matrix
$\PL$ has top-right $(n-1) \times (n-1)$ submatrix 
\smash{$P(\deltac{\sfrac{1}{1-\mu}}{2})$}. In particular, $\lambda_x = 
x\mu - (x-1)$ for $x \in n-1$ and, by~\eqref{eq:DSdelta2}, we have 
$\DSL_{n-3,0} = \DS(\deltac{\sfrac{1}{1-\mu}}{2})_{n-3,0} = 0$. 
Hence $\PL_{n-3,n-1} = 0$ and using the reversibility hypothesis, it follows from
Lemma~\ref{lemma:reversibilityCriteria}(i) that $\PL_{n-1,n-3} = 0$.
But by Lemma~\ref{lemma:practical}(iv), 
\[ \PL_{n-1,n-3} = \binom{n-1}{n-3} (\lambda_{n-3} - 2\lambda_{n-2}
+ \lambda_{n-1}).\] 
Therefore
\[ \bigl( (n-3)\mu - (n-4) \bigr) - 2\bigl( (n-2)\mu - (n-3) \bigr) + \lambda_{n-1} = 0 \]
which simplifies to $-(n-1)\mu + (n-2) + \lambda_{n-1} = 0$. Hence $\lambda_{n-1} = (n-1)\mu - (n-2)$ which
is \smash{$\lambdac{\sfrac{1}{1-\delta}}{2}{n-1}$} by~\eqref{eq:DSdelta2}. Therefore
by~\eqkey, \smash{$\PL = P(\deltac{\sfrac{1}{1-\mu}}{2})$}. Thus $a' = \mfrac{1}{1-\mu} = a'_2(\nu)$.
Moreover, by  Lemma~\ref{lemma:deltaDomainSimplified}, $2 > \lfloor \mfrac{1-\mu}{\mu} (n-2) \rfloor + 2$.
Hence $\mfrac{1-\mu}{\mu}(n-2) < 1$ which implies $n-1 < \mfrac{1}{1-\mu}$, 
hence $n \le \lceil a'_2(\nu) \rceil$,
giving the final requirement for (iv).
\end{proof}

\subsubsection*{Inductive step}
Finally we are ready to prove Proposition~\ref{prop:reversibleImpliesGamma}.
Suppose that $n \ge 4$.
By induction and the global  reversibility hypothesis, there is a
weight $\gamma$ of the required binomial type such that $\DSL$ 
agrees with $H(\gamma)$ except perhaps in its bottom row. 
Moreover, the parameters for $\gamma$ (namely~$a$ and~$b$
when $\gamma=\gammac{a}{b}$, $c$ when $\gamma = \gammai{c}$ and $a'$, $b'$ when $\gamma = \deltac{a'}{b'}$) are unique.
Restricting
further to the top-right $3 \times 3$ submatrix, the base case implies that $\mu > \nu \ge 2\mu - 1$.
If $\nu = 2\mu - 1$ then we are done
by Lemma~\ref{lemma:nuLeast}. 
Therefore
we may assume that $1 > \mu > \nu > 2\mu - 1$.

Consider the $4$-cycle 
\[ n-1 \mapsto 1 \mapsto n-2 \mapsto 2 \mapsto n-1. \]
(If $n=4$ then instead take the $3$-cycle $3\mapsto 1 \mapsto 2 \mapsto 3$ obtained by deleting the repeated vertex $2$.)
By the version of
Kolmogorov's Criterion in Lemma~\ref{lemma:reversibilityCriteria}(ii), and the hypothesis that $\PL$ is reversible,
we have 
\begin{equation}
\label{eq:cycle1} \PL_{n-1,1}\PL_{1,n-2}\PL_{n-2,2}\PL_{2,n-1} = \PL_{n-1,2}\PL_{2,n-2}\PL_{n-2,1}\PL_{1,n-1}. \end{equation}
(This holds when $n=4$ since we have simply introduced an extra term of $\PL_{2,2}$.)
By the inductive assumption
$\PL_{xz} = P(\gamma)_{xz}$ for all $x \in n-1$. 
Moreover by Theorem~\ref{thm:spectrum}  each $\gamma$-weighted involutive walk is reversible, and so~\eqref{eq:cycle1} holds when every $\PL_{xz}$ is replaced with $P(\gamma)_{xz}$. That is,
\begin{equation}
\label{eq:cycle1a} \begin{split}
P(\gamma)_{n-1,1}P(\gamma)_{1,n-2}&P(\gamma)_{n-2,2}P(\gamma)_{2,n-1} \\
&= P(\gamma)_{n-1,2}P(\gamma)_{2,n-2}P(\gamma)_{n-2,1}P(\gamma)_{1,n-1}. \end{split} \end{equation}
Let $q$ be the common value in~\eqref{eq:cycle1a}. Note that $q \not= 0$, since $\mu > 2\nu-1$ and so
$\gamma$ is not the exceptional
weight $\deltac{\sfrac{1}{1-\mu}}{2}$, the matrix $P(\gamma)$ has at least $3$ non-zero anti-diagonals and
all terms in the product are non-zero. Thus
\begin{align*} \PL_{1,n-2}\PL_{n-2,2}\PL_{2,n-1} &= P(\gamma)_{1,n-2}P(\gamma)_{n-2,2}P(\gamma)_{2,n-1} = q/P(\gamma)_{n-1,1} \\
\PL_{2,n-2}\PL_{n-2,1}\PL_{1,n-1} &= P(\gamma)_{2,n-2}P(\gamma)_{n-2,1}P(\gamma)_{1,n-1} = q/P(\gamma)_{n-1,2}.
\end{align*}
By Lemma~\ref{lemma:practical}(iv),  we have $\PL_{n-1,1} = \binom{n-1}{1}(\lambda_{n-2}-\lambda_{n-1})$
and $\PL_{n-1,2} = \binom{n-1}{2}(\lambda_{n-3}-2\lambda_{n-2} + \lambda_{n-1})$.
Again the analogous equation holds for $P(\gamma)$. Moreover, by the induction hypothesis and~\eqkey,
$\lambda_x = \DSL_{xx} = \DS(\gamma)_{xx} = \lambda_x(\gamma)$ for $x \in n-1$. Therefore
we obtain another pair of equations
\begin{align*}
\PL_{n-1,1} &= P(\gamma)_{n-1,1} - \binom{n-1}{1}\bigl( \lambda_{n-1} - \lambda_{n-1}(\gamma) \bigr) \\
\PL_{n-1,2} &= P(\gamma)_{n-1,2} + \binom{n-1}{2}\bigl( \lambda_{n-1} - \lambda_{n-1}(\gamma) \bigr).
\end{align*}
Writing $\Delta$ for $\lambda_{n-1} - \lambda_{n-1}(\gamma)$ and making the substitutions in~\eqref{eq:cycle1}
indicated by the four displayed equations above
we obtain
\[ \bigl( P(\gamma)_{n-1,1} - (n-1) \Delta \bigr) \frac{q}{P(\gamma)_{n-1,1}} = \bigl( P(\gamma)_{n-1,2}
+ \binom{n-1}{2}\Delta \bigr) \frac{q}{P(\gamma)_{n-1,2}}. \]
Using that $P(\gamma)_{n-1,z} = \gamma_{[z^\star,n-1]}/N(\gamma)_{n-1}$ and cancelling all common terms, including
$N(\gamma)_{n-1}$, we obtain
\[ -(n-1) \frac{\Delta}{\gamma_{[n-2,n-1]}} = \binom{n-1}{2} \frac{\Delta}{\gamma_{[n-3,n-1]}} \]
or equivalently
\begin{equation}\label{eq:cycle2}
\bigl( (n-1) \gamma_{[n-3,n-1]} + \binom{n-1}{2} \gamma_{[n-2,n-1]} \bigr) \Delta = 0.
\end{equation}
If $\gamma$ is the weight $\gammai{c}$ then $\gammai{c}_{[y,n-1]} = \binom{n-1}{y} c^{n-1-y}$
and so~\eqref{eq:cycle2} becomes
\smash{$(n-1)\binom{n-1}{2}c \bigl(c-1) \Delta = 0$}.
Since $c > 1$, we have $\Delta = 0$. Hence $\lambda_{n-1} = \lambda(\gammai{c})_{n-1}$ and by~\eqkey\ we get (ii)
in Proposition~\ref{prop:reversibleImpliesGamma}.
If $\gamma$ is the weight $\gammac{a}{b}$ then $\gammac{a}{b}_{[y,n-1]} = \binom{a+y}{y} \binom{b+n-1-y}{n-1-y}$
and so~\eqref{eq:cycle2} becomes
\[ \Bigl( (n-1) \binom{a+n-3}{n-3}\binom{b+2}{2}  + \binom{n-1}{2} \binom{a+n-2}{n-2}\binom{b+1}{1} \Bigr) \Delta = 0. \]
Using $\binom{m+1}{c+1} = \binom{m}{c} \frac{m+1}{c+1}$ this simplifies to
\[ \mfrac{1}{2} \bigl( (n-1) \binom{a+n-3}{n-3} \binom{b+1}{1} \Bigl( \frac{b+2}{2} + \frac{n-2}{2} \frac{a+n-2}{n-2}\Bigr)
\Delta = 0,\] 
and hence to $(a+b+n) \Delta = 0$. Since $a$, $b \in \Rg$ and $n \ge 4$, again we get $\Delta = 0$, and 
by~\eqkey\ we get~(i).
Finally suppose that $\gamma$ is the weight
$\deltac{a'}{b'}$ and so $\gamma_{[y,n-1]} = \binom{a'-1}{y}\binom{b'-1}{n-1-y}$.
Substituting in~\eqref{eq:cycle2} we get
\[ \Bigl( (n-1) \binom{a'-1}{n-3}\binom{b'-1}{2} + \binom{n-1}{2} \binom{a'-1}{n-2}\binom{b'-1}{1} \Bigr)\Delta =0 \]
which simplifies by $\binom{m}{c+1} = \binom{m}{c} \mfrac{m-c}{c+1}$ to
\[ (n-1) \binom{a'-1}{n-3} \binom{b'-1}{1} \Bigl( \frac{b'-2}{2} + \frac{n-2}{2} \frac{a'-n+2}{n-2} \Bigr) \Delta = 0 \]
and hence to $(a'+b'-n)\Delta = 0$. 
In either inductive case we have $a' \ge n-1$ and $b' \ge 2$, hence $a'+b' - n > 0$ and so $\Delta = 0$.
By~\eqkey\ we get $\PL = P(\deltac{a'}{b'})$. 

Since $0$ is accessible from every state, 
Lemma~\ref{lemma:deltaDomainSimplified} implies that $n \le \lceil a' \rceil$. 
If the inductive case is~(iv) then~(iv) still holds: $\mu > \mfrac{1}{2}$,
$\nu = \nu_m(\mu)$ and $a' = a'_m(\mu)$ are inductive assumptions, and $m \ge \lfloor \mfrac{1-\mu}{\mu} (n-2) \rfloor + 2$
and $n \le \lceil a'_m(\nu) \rceil$ follow from Lemma~\ref{lemma:deltaDomainSimplified}.
Finally suppose the inductive case is~(iii), so $\nu_{n-2}(\mu) < \nu < \mu^2$. If we have the stronger inequality
$\nu_{n-1}(\mu) < \nu$ then~(iii) still holds. In the remaining case (iii) holds with $\nu_{n-2}(\mu) < \nu \le \nu_{n-1}(\mu)$.
Since $b'_\mu$ is an increasing function of $\nu$, we have $n-2 < b' \le n-1$. If $n-2 < b' < n-1$ then
\smash{$\deltac{a'}{b'}_{[0,n-1]} = \binom{b'-1}{n-1} < 0$} contradicts that $P(\deltac{a'}{b'})$ is equal to the stochastic
matrix $\PL$. Therefore $b' = n-1$, $\nu = \nu_{n-1}(\mu)$ and $a' = a'_{n-1}(\mu)$. (This is the case where
$P(\deltac{a'_{n-1}(\mu)}{n-1})$ has $n-1$ non-zero anti-diagonal bands and a zero in its bottom right corner.)
By Lemma~\ref{lemma:deltaDomainSimplified}, $n-1 > \mfrac{1-\mu}{\mu}n + 2 - \mfrac{1}{\mu}$. If $\mu \le \mfrac{1}{2}$ then
the right-hand side is at least $n$, a contradiction. Therefore $\mu > \mfrac{1}{2}$ and we have shown that (iv) holds.
This completes the proof.

\section{The involutive walk on $[0,1]$}\label{sec:Hilbert}

We now turn to involutive walks defined on the real interval 
$[0,1] = \{x \in \R : 0 \le x \le 1\}$ with its usual order
and anti-involution $x^\star = 1 -x$. From now on $[y,x]$ denotes a subinterval of $[0,1]$.

\begin{definition}\label{defn:realWeight}
A \emph{real weight} is a function $\gamma$ on the set of non-empty intervals of $[0,1]$
taking values in $\R^{\ge0}$ such that, for each $x \in [0,1]$,
the function $y \mapsto \gamma_{[y,x]}$ is continuous and non-zero almost everywhere in $[0,x]$,
and $\int_0^x \gamma_{[y,x]} \,\mathrm{d}y < \infty$.
\end{definition}

Thus a real weight defines a probability distribution on each interval $[0,x]$ having
probability density function $\gamma_{[y,x]}/N(\gamma)_x$ where
$N(\gamma)_x = \int_0^x \gamma_{[y,x]} \, \mathrm{d} y$. The positivity assumption in the definition
makes these weights the analogue of the strictly positive discrete weights in Definition~\ref{defn:weight}.
We define \emph{atomic} and \emph{$\star$-symmetric} for real weights as in this definition.
The $\gamma$-weighted involutive walk on $[0,1]$ is defined
as in Definition~\ref{defn:involutiveWalk}, replacing $\{0,1,\ldots,n-1\}$ with $[0,1]$.

\subsection{Invariant distributions}

\begin{proposition}\label{prop:invariantFactorReal}
Let $\alpha$ be an atomic real weight and let
$\beta$ be a $\star$-symmetric real weight.
Suppose that $\int_0^1 \alpha_{1-x} N(\alpha \beta)_x\,\mathrm{d}x < \infty$. 
Then the $\alpha\beta$-weighted involutive walk on $n$ is reversible with
respect to an invariant distribution proportional to $\alpha_{1-x} N(\alpha \beta)_x$.
\end{proposition}

\begin{proof}
It is routine to solve the detailed balance equation using essentially the same
arguments used to prove Lemma~\ref{lemma:invariant} and Lemma~\ref{lemma:pointwiseProductReversible}.
\end{proof}

The integrability assumption in the lemma is essential.

\begin{example}\label{ex:realWeight}
Let $a \in \R$ with $a > 0$. 
Let $\alpha_{[y,x]} = 1/y^a(1-y)^{a+1}$ and let $\beta_{[y,x]} = y^a (1-x)^a$ for $0 < y \le x$.
Observe that $\alpha$ and $\beta$ are real weights and that $\alpha$ is atomic and $\beta$ is $\star$-symmetric. 
Let $\gamma = \alpha\beta$, so $\gamma_{[y,x]} = (1-x)^a/(1-y)^{a+1}$.
We have
\[ N(\gamma)_x = \int_0^x \frac{(1-x)^a}{(1-y)^{a+1}} \, \mathrm{d} y 
= \frac{1}{a} \bigl( 1- (1-x)^{a} \bigr)\]
hence $\gamma$ is a real weight. 
However 
\[ \alpha_{1-x}N(\gamma)_{x} = \frac{1}{a x^{a+1}}\Bigl(\frac{1}{(1-x)^a} -1  \Bigr)
\]
is not integrable on any interval containing $1$, and not integrable on any interval
containing $0$ whenever $a \ge 1$.
\end{example}

For the remainder of this section, let $\theta : [0,1] \rightarrow \R^{\ge 0}$ be the function proportional
to the invariant distribution defined by $\theta_x = \alpha_{1-x} N(\alpha)_x$.
Note that, by Definition~\ref{defn:realWeight}, $\theta_x$ is non-zero almost everywhere.
Given Example~\ref{ex:realWeight}, we also assume that $\sqrt{\theta_x}$ is in the Hilbert space $\L^2[0,1]$
of square-integrable functions on $[0,1]$.

\subsection{Hilbert spaces}\label{subsec:Hilbert}
In this section we describe a setting in which the analogues of
left- and right-multiplication by the transition matrix for a discrete involutive
walk are compact linear operators on Hilbert space. 
Our~\S\ref{subsec:eigenvectors}
and \cite{BoydDiaconisParriloXiao} suggest the correct inner-product spaces in which to work.
Let $\alpha$ be an atomic real weight and let $\beta$ be a $\star$-symmetric real weight. Define
$K:[0,1]^2 \rightarrow \R$ by
\begin{equation}\label{eq:K} K(x,z) = [x+z\ge 1]\frac{\sqrt{\alpha_{1-z}}\sqrt{\alpha_{1-x}} \beta_{[1-z,x]}}{\sqrt{N(\alpha\beta)_{z}}
\sqrt{N(\alpha\beta)_x}} \end{equation}
where $[x+z \ge 1]$ is an Iverson bracket.
When
$K(x,z) \in \L^2\bigl( [0,1]^2 \bigr)$,
the
integral operator $M : \L^2([0,1]) \rightarrow \L^2([0,1])$ defined by
$(Mf)(x) = \int_0^1 K(x,z) f(z) \, \mathrm{d} z$ is compact.
(See for instance \cite[Theorem~8.8]{YoungHilbert}.)
Since $K(x,z) = K(z,x)$ for all $x$, $z \in [0,1]$, 
$M$ is also self-adjoint.

\begin{definition}\label{defn:H}
Let $\H_L = \{f(x) / \sqrt{\theta_x} : f \in \L^2([0,1])\}$ with inner product $\langle \ , \ \rangle_{\H_L}$
defined by $\langle f, g \rangle = \int_0^1 \theta_x f(x)g(x) \, \mathrm{d} x$.
Let $\H_R = \{f(x) \sqrt{\theta_x} : f \in \L^2([0,1])\}$ with inner product $\langle \ , \ \rangle_{\H_R}$ defined by 
 $\langle f, g \rangle = \int_0^1 f(x)g(x)/\theta_x \, \mathrm{d} x$.
\end{definition}

The maps sending $f \in \L^2([0,1])$ to $f/\sqrt{\theta_x} \in \H_L$ 
and $f\sqrt{\theta_x} \in \H_R$ are isometric isomorphisms from $\L^2([0,1])$ to
$\H_L$ and $\H_R$, respectively. (The first map is well-defined 
because the zeros of $\theta_x$ form
a null-set.) Therefore $\H_L$ and $\H_R$ are Hilbert spaces.
Moreover the constant function is in $\H_L$ and, by our integrability assumption, $\theta_x$ is in $\H_R$.
The analogues of left- and right-multiplication by the transition matrix are the integral operators
$L_P : \H_L \rightarrow \H_L$ and
$R_P : \H_R \rightarrow \H_R$ defined by
$(L_Pg)(x) = M(g(z) \sqrt{\theta_z})/\sqrt{\theta_x}$ for $g \in \H_L$
and $R_P(g)(z) = M(g(x) / \sqrt{\theta_x}) \sqrt{\theta_z}$ for $g \in \H_R$.
(Note we deliberately use different variables: for left-multiplication
we integrate over the final state~$z$; for right-multiplication we integrate
over the initial state~$x$.)
Equivalently,
\begin{align} (L_Pg)(x) &= 
\int_{1-x}^1 \frac{\alpha_{1-z} \beta_{[1-z,x]}}{N(\alpha\beta)_x}  g(z)\, \mathrm{d} z 
\label{eq:LP} \\
(R_Pg)(z) &= \int_{1-z}^1 \frac{\alpha_{1-z} \beta_{[1-z,x]}}{N(\alpha\beta)_x}g(x) \,
\mathrm{d} x. \label{eq:RP} \end{align}
Since $L_P$ and $R_P$ are conjugate by isometric isomorphisms to the self-adjoint compact operator $M$,
they are also self-adjoint and compact, and have the same eigenvalues.

We can now give a convenient sufficient condition for the real involutive walk to converge. 
The proof uses the generalization of the Perron--Frobenius Theorem to operators on Banach spaces. 
Let $\pi_x = C \theta_x$ be the invariant probability distribution given by Proposition~\ref{prop:invariantFactorReal}.

\begin{proposition}\label{prop:ergodicReal}
The invariant distribution $\pi_x$ 
is unique and
the $\alpha\beta$-weighted involutive walk started at any probability
distribution $\theta \in \H_R$ converges
to $\pi_x \in \H_R$.
\end{proposition}

\begin{proof}
The condition in Definition~\ref{defn:realWeight} that $\gamma_{[y,x]} > 0$ for almost
all $y \in [0,x]$ implies that if $C$ is the cone in $\mathcal{H}_R$ of non-negative functions
then $R_P : C \backslash \{0\} \rightarrow C^\circ$, where $C^\circ$ is the interior of $C$.
Hence, by the 
Krein--Rutman Theorem (see \cite[Theorem 6.3]{KreinRutman}), 
$1$ is the unique largest eigenvalue  (in modulus) of $R_P$,
and $\pi_x$ is the unique $1$-eigenfunction. Since $R_P$ is compact and self-adjoint,
there is a complete orthonormal basis of $\H_R$ of eigenfunctions; expanding $\theta$ in this
basis and then applying $R_P^t$ shows that $R_P^t \theta$ converges, in $\H_R$, to the invariant
distribution $\pi$, as $t \rightarrow \infty$.
\end{proof}

\subsection{Trigonometric example}\label{subsec:trigonometricExample}

As an illustrative example, to which we return in the final subsection, we take 
$\alpha_x = \sin \pi x$ for $x \in [0,1]$ and $\beta_{[y,x]} = 1$ for each non-empty interval
$[y,x]$. Note that $\alpha$ and $\beta$ are real weights, $\alpha$ is atomic and $\beta$ is symmetric. 
We have $N(\alpha)_x = \int_0^x \sin \pi y \,\mathrm{d} y = 
\mfrac{1}{\pi}(1-\cos \pi x)$. 
The hypothesis for Proposition~\ref{prop:invariantFactorReal}, that $\sqrt{\theta_x} \in \L^2[(0,1)]$ holds.
This proposition implies that the
$\alpha$-weighted involutive walk is reversible with invariant distribution proportional to
$\theta_x$ where $\theta_x = \alpha_{1-x}N(\alpha)_x = 
\mfrac{1}{\pi} (\sin \pi x) (1-\cos\pi x) = \mfrac{1}{\pi} \sin \pi x - \mfrac{1}{2\pi} \sin 2\pi x$.
Integrating over $[0,1]$, the first summand gives $\frac{2}{\pi^2}$ and the second $0$, hence
$\pi_x = \mfrac{\pi}{2} (\sin \pi x)(1-\cos \pi x)$ is an invariant distribution.
The kernel $K$ in~\eqref{eq:K} satisfies
\[ K(x,z)^2 = [x+z \ge 1] \frac{\sin \pi z \sin \pi x}{(1-\cos \pi x) (1-\cos \pi z)}. \]
Using $\sin \alpha / (1-\cos \alpha) = \cot \mfrac{\alpha}{2}$, this
simplifies to $K(x,z)^2 = [x+z \ge 1] \cot \mfrac{\pi}{2} z \cot \mfrac{\pi}{2} x$. We have
\[ \int_{(x,z) \in [0,1]^2 \atop x + z \ge 1 }  \cot \mfrac{\pi}{2} z \cot \mfrac{\pi}{2} x \, \mathrm{d}z \,\mathrm{d}x
= \int_0^1 \!\!\int_0^1 x\tan \mfrac{\pi}{2}xu \cot \mfrac{\pi}{2} x \, \mathrm{d}u \, \mathrm{d}x
\]
by the substitution $1-z = xu$. 
Since $x \tan \mfrac{\pi}{2} xu  \cot \mfrac{\pi}{2} x
\le x \tan \mfrac{\pi}{2} x \cot \mfrac{\pi}{2} x = x$, the integrand
is bounded. (In fact it is even continuous, except at $(1,1)$.)
Therefore $K \in \L^2\bigl( [0,1]^2 \bigr)$ and the general theory in the previous subsection applies.
In particular, we may  apply Proposition~\ref{prop:ergodicReal} to conclude that $\pi_x$
is the unique invariant distribution for the $\alpha$-weighted
involutive walk on $[0,1]$, and that the process converges, in the Hilbert space sense, to the invariant
distribution $\pi_x$.

We now find the spectra of $L_P$ and $R_P$. 
In~\S\ref{sec:gammaInfinityWalk} and \S\ref{sec:gamma} we used right-eigenvectors,
corresponding to left-multiplication by the transition matrix, and again here it
is most convenient to work with $L_P : \H_L \rightarrow \H_L$, defined as in~\eqref{eq:LP} by
\[ (L_Pg)(x) = \pi \int_{1-x}^1 \frac{\sin \pi z }{1-\cos \pi x} g(z) \, \mathrm{d} z,\]
where $\H_L = \{f(x)/\sqrt{\theta_x} : f \in \L^2([0,1])\}$.
For $n \in \N_0$, let $V_n$ be the $n$-dimensional subspace of~$\H_L$ 
spanned by $1, \cos \pi x$, $\cos 2\pi x$, $\ldots$, $\cos \pi(n-1)x$
and let $V_\infty = \bigcup_{n=0}^\infty V_n$. 
The image of multiplication by $\sqrt{\pi_x}$ on $\L^2([0,1])$ is dense in $\L^2([0,1])$.
(For example, it contains any step function supported on a closed subinterval of $(0,1)$.)
Since $\sqrt{2}, \sqrt{2}\cos \pi x, \sqrt{2}\cos 2\pi x, \ldots $ is an orthonormal 
Hilbert space basis of $\L^2([0,1])$,
the set $\sqrt{\pi_x}, \sqrt{\pi_x} \cos \pi x$, $\sqrt{\pi_x} \cos 2\pi x$, $\ldots $
is dense in $\L^2([0,1])$. Hence, dividing by $\sqrt{\pi_x}$, we 
see that $V_\infty$ is dense in $\H_L$.


\begin{lemma}\label{lemma:cosAlt}
For any $d \in \N_0$ we have
\[ L_P (\cos^d \hskip-0.5pt\pi x) \in \frac{(-1)^d}{d+1} \cos^d \hskip-0.5pt\pi x + V_d. \]
\end{lemma}

\begin{proof}
We have
\begin{align*}
L_P (\cos^d \hskip-0.5pt \pi x) 
&= \frac{\pi}{1-\cos \pi x} \int_{1-x}^1 \sin \pi z \cos^d \hskip-0.5pt \pi z \, \mathrm{d} z
\\
&= \frac{\pi}{1-\cos \pi x} \frac{-1}{\pi (d+1)} (\cos^{d+1}\hskip-0.5pt \pi z) \Bigr|_{1-x}^1 \\
&= \frac{-1}{d+1} \frac{(-1)^{d+1}- \cos^{d+1} \pi (1-x)}{1- \cos \pi x} \\
&= \frac{(-1)^d}{d+1} \frac{1- \cos^{d+1} \pi x}{1- \cos \pi x} \\
&= (-1)^d \frac{1 + \cos \pi x + \cdots + \cos^d \hskip-0.5pt\pi x}{d+1} 
\end{align*}
as required.
\end{proof}

Since $\langle 1, \cos \pi x, \ldots, \cos^d \pi x \rangle = \langle 1, \cos \pi x, \ldots, \cos d \pi x \rangle$
for each $d$, it follows that
$L_P$ has eigenvalues $(-1)^{d}/(d+1)$ for $d \in \N_0$. Moreover, since $V_\infty$ is dense in $\mathcal{H}_L$, the eigenfunctions form a complete
orthonormal basis for~$\mathcal{H}_L$. The convergence of the involutive walk is determined by the second largest
eigenvalue (in modulus), namely $-1/2$.



\begin{remark}
We remark that a similar method to Lemma~\ref{lemma:cosAlt} gives
part of the spectrum of $L_P$ whenever $\alpha_{1-x} = \alpha_x$ for all $x \in [0,1]$ and so
$\alpha_{1-z} N(\alpha)^d_z$ is the derivative of $\mfrac{1}{d+1} N(\alpha)^{d+1}_z$.
But one cannot expect in general that the eigenfunctions obtained in this way will
be complete for the Hilbert space~$\H_L$.
\end{remark}

%
%

%

\section{The polynomially-weighted involutive walk on the interval}\label{sec:intervalWalk}

Fix, throughout this section, $a$, $b \in \N_0$.
Observe that the weight
$\gammac{a}{b} = \binom{y+a}{y} \binom{b+x-y}{x-y}$ 
defined in~\eqref{eq:gamma}
is asymptotically proportional to $y^a (x-y)^b$ when $y$ and $x-y$ are both large.
Let $\kappac{a}{b}$ be the real weight defined by
\smash{$\kappac{a}{b}_{[y,x]} = y^a (x-y)^b$}.
In this section we use the theory from \S\ref{sec:Hilbert} to
 show that the $\kappac{a}{b}$-weighted involutive walk, is, in a precise sense,
the continuous limit of the $\gammac{a}{b}$-involutive walk. In particular, it has a discrete
spectrum and its eigenvalues
are given by letting~$n$ tend to infinity in Theorem~\ref{thm:spectrum}.
We end with some final remarks comparing the discrete and continuous cases.

\subsection{Polynomial weights}

We begin with the analogues of Lemmas~\ref{lemma:mbinom} and~\ref{lemma:chainWeightsTotal}.
Let $P(\kappac{a}{b})_{xz}$ be the probability density of a step from $x$ to $z$.
Part (i) below is a standard result related to the beta function which we prove to make the paper more self-contained.

\begin{lemma}\label{lemma:ctsBasic} We have
\begin{thmlist}
\item $\int_0^1 w^a (1-w)^b \d w = (a+b+1)^{-1} \binom{a+b}{a}^{-1}$;
\item $N(\kappac{a}{b})_x = x^{a+b+1}(a+b+1)^{-1} \binom{a+b}{b}^{-1}$;
\item $\displaystyle
P(\kappac{a}{b})_{xz} = [x+z \ge 1](a+b+1)\binom{a+b}{b} \frac{(1-z)^a(x+z-1)^b}{x^{a+b+1}}$
where $[x+z \ge 1]$ is an Iverson bracket.
\end{thmlist}
\end{lemma}

\begin{proof}
By integrating $(w + (1-w)t)^c = \sum_{b=0}^c \binom{c}{b} w^{c-b} (1-w)^b t^b$ over $w$ we get
\[ \frac{1-t^{c+1}}{(c+1)(1-t)} = \sum_{b=0}^c \binom{c}{b}t^b \int_0^1 w^{c-b} (1-w)^b \d y. \]
Since the left-hand side is $\frac{1}{c+1}(1+t + \cdots + t^c)$,  (i)
follows by setting $c=a+b$ and comparing coefficients of $t^b$.
Part (ii) follows from (i) using the change of variables $xw = y$ to write
$\int_0^x y^a (x-y)^b \d y = x^{a+b}\! \int_0^1 w^a (1-w)^b x \d w$.
Now, starting at~$x \in [0,1]$, we may step to $z \in [0,1]$ if and only if $1-z \le x$; hence 
\smash{$P(\kappac{a}{b})_{xz} = [1-z \le x] \kappac{a}{b}_{[1-z,x]}/N(\kappac{a}{b})_x$}.
Part (iii) now follows from the definition of $\kappac{a}{b}_{[1-z,x]}$ and (ii).
\end{proof}

Set $\pi_x = C (1-x)^a x^{a+b+1}$
where $C = (2a+b+2)\binom{2a+b+1}{a}$ is the normalization factor given by Lemma~\ref{lemma:ctsBasic}(i).

\begin{lemma}\label{lemma:ctsInvar}
The $\kappac{a}{b}$-interval involutive walk has unique invariant probability density function 
$\pi_x$ and is reversible.
\end{lemma}

\begin{proof}
This is immediate from Proposition~\ref{prop:invariantFactorReal}.
%
\end{proof}

\subsection{Spectrum}
We now find the spectrum of the $\kappac{a}{b}$-weighted involutive walks. The Hilbert space $\H_L$ defined
in~\S\ref{subsec:Hilbert} is as follows.

\begin{definition}
Let $\H_L = \{f(x) / \sqrt{\pi_x} : f \in \L^2([0,1]) \}$, with inner product $\langle \ , \ \rangle_{\H_L}$
defined by $\langle f, g \rangle = \int_0^1 (1-x)^a x^{a+b+1} f(x) g(x) \d x$.
\end{definition}

For $e \in \N$, let $V_e$ be the subspace of $\H_L$ of polynomials of degree
strictly less than~$e$, and let $V_{\infty} = \bigcup_{e \in \N} V_e$.
Thus each $V_e$ is a closed linear subspace of $\H_L$
of dimension~$e$ and
by a similar argument to~\S\ref{subsec:trigonometricExample}, $V_\infty$ is a dense subspace of $\H_L$.
By~\eqref{eq:LP}, the analogue of left-multiplication
by the transition matrix~is $L_P : \H_L \rightarrow H_L$, defined by
\begin{equation}
\label{eq:LPp} (L_Pf)(x) = (a+b+1)\binom{a+b}{a}
\int_{1-x}^1 \frac{(1-z)^a (x+z-1)^b}{x^{a+b+1}} f(z) \d z.
\end{equation}
We also need the analogous operator for left-multiplication
by the transition matrix for down-steps,
\begin{equation}
\label{eq:LHp} (L_Hf)(x) = (a+b+1)\binom{a+b}{a}
\int_{0}^x \frac{y^a (x-y)^b}{x^{a+b+1}} f(y) \d y. 
\end{equation}
Changing
variables from $y$ to $1-z$ in the definition of $L_Hf$ shows that
$L_P = L_H J$, where $J : \H_L \rightarrow \H_L$ is the self-adjoint involution
defined by $(Jf)(x) = f(1-x)$, analogous to the matrices $J(n)$ used earlier. 
Since $J$ and $L_P$ preserve $\mathcal{H}_L$, the same holds for $L_H$.
Recall from Theorem~\ref{thm:spectrum} 
that the eigenvalues of the discrete $\gammac{a}{b}$-weighted involutive walk on $\{0,\ldots, n-1\}$ are
$(-1)^d\lambdac{a}{b}{d}$ for $0 \le d < n$, where
$\lambdac{a}{b}{d} = \binom{a+d}{d}/ \binom{a+b+d+1}{d}$. 

\begin{lemma}\label{lemma:ctsHtransform}
Let $d \in \N_0$. We have $L_H(x^d) = \lambdac{a}{b}{d} x^d$.
\end{lemma}

\begin{proof}
By the substitution $y = xw$ in~\eqref{eq:LHp}
and Lemma~\ref{lemma:ctsBasic}(i) we get
\begin{align*}
L_H(x^d) &= (a+b+1)\binom{a+b}{a} \int_0^x \frac{y^a (x-y)^b}{x^{a+b+1}} y^d \, \mathrm{d} y \\
&= (a+b+1) \binom{a+b}{a} x^d \int_0^1 w^{a+d}(1-w)^b \, \mathrm{d} w \\
&= \frac{(a+b+1)\binom{a+b}{a} }{(a+b+d+1)\binom{a+b+d}{a+d} } x^d .\end{align*}
Now use
$\binom{a+b+d}{a+d} / \binom{a+b}{a} = \binom{a+b+d}{d} / \binom{a+d}{a}$
to see that the eigenvalue is as claimed.
\end{proof}

We now show that the analogue of Proposition~\ref{prop:practical}(i) holds.

\begin{lemma}\label{lemma:ctsPtransform}
Let $d \in \N$. Then
\[ L_P (x^d) \in (-1)^d\lambdac{a}{b}{d}x^d + V_d. \]
\end{lemma}

\begin{proof}
Since $L_P = L_H J$ and $(L_H J)(x^d) = L_H((1-x)^d) \in (-1)^d x^d + \H_d$,
this is immediate from Lemma~\ref{lemma:ctsHtransform}.
 \end{proof}
 
We can now give a complete description of the spectrum and eigenfunctions for $L_P$.
Let $g_0(x), g_1(x), g_2(x), \ldots $ be the functions obtained by Gram--Schmidt orthonormalization
applied to $1$, $x$, $x^2, \ldots $, regarded as elements of the Hilbert space $\H_L$.

\begin{theorem}\label{thm:ctsLP}
The operator $L_P : \H_L \rightarrow \H_L$ 
is compact and self-adjoint. It has eigenvalues $(-1)^d\lambdac{a}{b}{d}$ 
for $d \in \N_0$. The normalized eigenfunction for $(-1)^d\lambdac{a}{b}{d}$ is $g_d$.
The eigenfunctions $g_0, g_1, \ldots $ are an orthonormal Hilbert space basis
for~$\H$ and $L_P(f) = \sum_{d=0}^\infty (-1)^d\lambdac{a}{b}{d} \langle g_d, f \rangle g_d $
for each $f \in \H$.
\end{theorem}

\begin{proof}
Let $f$, $g \in \H_L$. Let $D = (a+b+1)\binom{a+b}{a}$. Since
\begin{align*} \langle f , L_P g \rangle_H
&= 
D\int_0^1 (1-x)^a x^{a+b+1} f(x) \int_{1-x}^1 \frac{(1-z)^a (x+z-1)^{a+b+1}}{x^{a+b+1}} g(z) \d z \\
&= D \int\!\!\! \int_{(x,z) \in [0,1]^2 \atop x+z \ge 1} (1-x)^a (1-z)^a (x+z-1)^{a+b+1} f(x)g(z) \d x \d z 
\end{align*}
is symmetric with respect to swapping $f$ and $g$ and $x$ and $z$, we have
$\langle f, L_P g \rangle_{\H_L} = \langle L_P f, g\rangle_{\H_L}$.
By induction on $d$, it follows from
Lemma~\ref{lemma:ctsPtransform} that 
for each $d \in \N_0$, there is a unique polynomial 
$g \in V_\infty$ of degree $d$
such that $\langle g, g\rangle_{\H_L} = 1$ and $L_P g = (-1)^d\lambdac{a}{b}{d} g$.
Since the eigenfunctions for distinct eigenvalues of $L_P$ are orthogonal,
we have $g = g_d$. Since the $g_d$ form a Hilbert space basis for ${\H_L}$, 
we have the claimed spectral decomposition of~$L_P$. 
It is clear that~\eqref{eq:chainEigenvalues} that  $\lambdac{a}{b}{d} \rightarrow 0$ as $d \rightarrow \infty$.
Therefore $L_P$ is compact.
\end{proof}

Dually, we defined in~\S\ref{subsec:Hilbert} the
Hilbert space $\H_R = \{f(x)\sqrt{\pi_x} : f \in \L^2([0,1]) \}$, with inner product
$\langle f, g\rangle_{\H_R} = 
\int_0^1 \frac{f(x)g(x)}{(1-x)^ax^{a+b+1}} \d x$. 
The map $S$ which sends $f \in \H_L$ to $(1-x)^a x^{a+b+1} f(x) \in {\H_R}$ is an isometric
isomorphism from $\H_L$ to ${\H_R}$.
Let $R_P : {\H_R} \rightarrow {\H_R}$ 
be the linear operator corresponding to right-multiplication by the transition
matrix of the interval involutive walk, defined by
\[ (R_P f)(x) = (a+b+1)\binom{a+b}{a}\int_{1-x}^1 \frac{(1-x)^a (x+z-1)^b}{z^{a+b+1}} f(z) \d z. \]
Recall that 
$\pi_x = C (1-x)^a x^{a+b+1}$.
Set $g_d'(x) = (1-x)^a x^{a+b+1}g_d$ for $d \in \N_0$.

\begin{theorem}\label{thm:ctsRP}
The operator $R_P : {\H_R} \rightarrow {\H_R}$ 
is compact and self-adjoint. It has eigenvalues $(-1)^d\lambdac{a}{b}{d}$ 
for $d \in \N_0$. The normalized eigenfunction for $(-1)^d\lambdac{a}{b}{d}$ is 
$g_d'$.
The eigenfunctions $g'_0, g'_1, \ldots $ are an orthonormal Hilbert space basis
for~$\mathcal{H}_R$ and $R_P(f) = \sum_{d=0}^\infty (-1)^d\lambdac{a}{b}{d} \langle g'_d, f \rangle g'_d $
for each $f \in {\H_R}$.
\end{theorem}

\begin{proof}
The operators $R_P$ and $L_P$ are conjugate by $S$, in that $R_P S = SL_P$.
The result is now immediate from
Theorem~\ref{thm:ctsLP}.
\end{proof}

Hence, started at any probability distribution in  ${\H_R}$, the $\kappa^{(a,b)}$-weighted
involutive walk converges, in the Hilbert space sense, to $\pi_x$. As in the discrete
case, the rate of convergence controlled
by the second largest eigenvalue in absolute value, namely $(a+1)/(a+b+2)$.

\begin{remark}
As one might expect from the discrete case,
all the results of this section extend to $a$, $b \in \R^{>-1}$, interpreting the factor $(a+b+1)^{-1}\binom{a+b}{b}^{-1}$ in the norm $N(\kappac{a}{b})$
as the value $\mathrm{B}(a+1,b+1)$ of the Beta integral. For instance, if $a=b=-\mfrac{1}{2}$ then
$N(\kappac{-\sfrac{1}{2}}{-\sfrac{1}{2}})_x = \int_0^x \frac{\mathrm{d}y}{\sqrt{y(1-y)}} = \pi = 
\mathrm{B}(\mfrac{1}{2}, \mfrac{1}{2}).$
\end{remark}

\subsection{Comparing the discrete and continuous}
\label{subsec:discreteContinuous}

\subsubsection*{Jacobi polynomials}
For $f, g \in \H_L$ we have 
\begin{align*}
\langle f, g \rangle_{\H_L} &= \int_{0}^1 (1-x)^a x^{a+b+1} f(x) g(x) \d x \\
&= 2^{-(2a+b+1)} \int_{-1}^1 (1-y)^a (1+y)^{a+b+1} f\bigl( \frac{1+y}{2} \bigr) g\bigl( \frac{1+y}{2}\bigr) \d y.
\end{align*}
It follows that the polynomials $g_d ( \frac{1+y}{2} )$ defined on $[-1,1]$
are orthonormal with respect to the inner product adapted from $\mathcal{H}_L$, namely
$\langle F, G \rangle = \int_{-1}^1 (1-y)^a (1+y)^{a+b+1} F(y)G(y) \d y$.
The families of orthogonal polynomials for weight functions of the form $(1-y)^A (1+y)^B$
are known as the \emph{Jacobi polynomials}. Thus
$g_d( \frac{1+y}{2})$, is up to a scalar, the Jacobi polynomial $J^{(a,a+b+1)}(y)$ of degree $d$ with 
parameters~$a$ and $a+b+1$.

For comparison,
in the discrete setting of \S\ref{subsec:eigenvectors},
for each fixed $n$ we described the right-eigenvectors $\rP{0}, \ldots, \rP{n-1}$ of the $n \times n$ matrix
$P(\gammac{a}{b})$ as the Gram--Schmidt orthonormalization
of the basis $v(0), \ldots, v(n-1)$ of $\mathbb{R}^n$ of columns of Pascal's Triangle,
with respect to the inner product $\langle v, w \rangle = \sum_{x=0}^{n-1} \pi_x v_x w_x$,
where $\pi_x$ is proportional to $\smbinom{n-x}{a}\smbinom{x+1}{a+b+1}$. Observe that the
scaling factor $\pi_x$ for $v_x w_x$ is asymptotically proportional to $(n-x)^a x^{a+b+1}$ when $x$ and $n-x$ are both large.
Thus, rescaling to the interval $[0,1]$ by setting $y = x/n$ and letting $n \rightarrow \infty$
we obtain $(1-y)^a (1+y)^{a+b+1}$, the same weight
function as for the Jacobi polynomials. 
The eigenvectors~$\rP{d}$  are therefore  the discrete
analogues of the Jacobi polynomials.
In particular,  if $\widehat{w}^{(n)}$ is the eigenvector $w(d) \in \R^n$ of $P(\gammac{a}{b})$,
scaled to have maximum entry (in absolute value) $1$,
and $\widehat{J}^{(a,a+b+1)}$ is the Jacobi polynomial $J_d^{(a,a+b+1)}$ with the same scaling imposed,
then $\widehat{w}^{(n)}_{nx} \rightarrow \widehat{J}^{(a,a+b+1)}_d(2x-1)$ as $n \rightarrow \infty$ for each $x \in [0,1]$.
In practice, the convergence is fast: see Figure~2 for an example.
The results of \S\ref{subsec:eigenvectors}
suggest these discrete analogues deserve further combinatorial study.

\begin{figure}
\begin{center}\setlength{\unitlength}{1cm}
\hspace*{-0.2in}\begin{picture}(5,5.5)
\put(-1,0) {\includegraphics[width=3.5in]{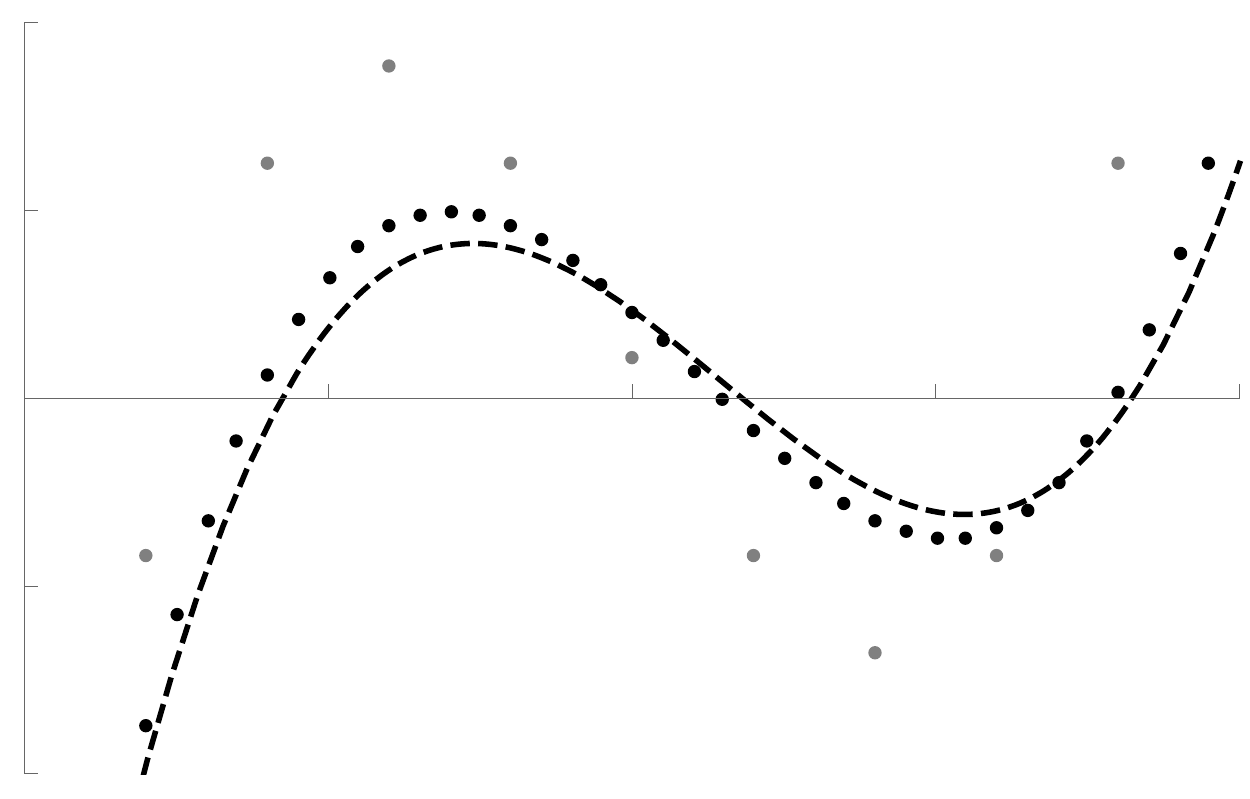}}
\put(-1.55,0.15) {$\scriptstyle -0.4$}
\put(-1.55,1.45) {$\scriptstyle -0.2$}
\put(-1.07,2.8) {$\scriptstyle 0$}
\put(-1.3,4.1) {$\scriptstyle 0.2$}
\put(-1.3,5.4) {$\scriptstyle 0.4$}
\put(1.1,2.575) {$\scriptstyle 0.25$}
\put(3.35,2.575) {$\scriptstyle 0.5$}
\put(5.4,2.575) {$\scriptstyle 0.75$}
\put(7.8,2.575) {$\scriptstyle 1$}
\end{picture}
\end{center}
\caption{The right-eigenvector $\rP{d}$ for $P(\gamma^{(a,b)})$ is the discrete analogue
of the Jacobi polynomial with parameters $a$ and $a+b+1$. The graph above shows 
the case $a=0$, $b=0$, rescaling to $[0,1]^2$, taking $n=10$ (grey dots) and $n=40$ (black dots) in the discrete case.}
\end{figure}
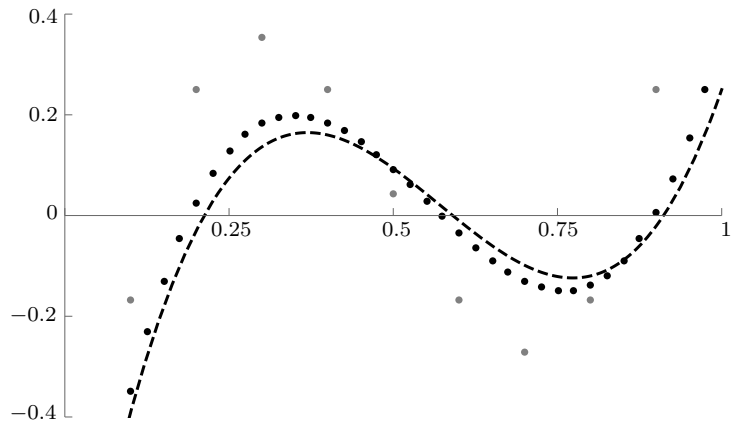

\subsubsection*{Anti-diagonal conjugators and anti-diagonal eigenbasis action}
We saw in \S\ref{sec:gammaSpectrum}
that the column~$v(d)$ of the binomial matrix $B(n)$ is a right-eigenvector of the $n \times n$
down-step matrix $H(\gammac{a}{b})$ for the $\gammac{a}{b}$-weighted involutive walk,
and, in this basis, the transition matrix 
$P(\gammac{a}{b})$ is upper-triangular with diagonal entries $(-1)^d \lambdac{a}{b}{d}$.
Theorem~\ref{thm:GB} shows that in this discrete case this behaviour (in the strong, global form) characterises
the binomial transforms~$\DSL$. Stated in a way we hope is intuitive, Proposition~\ref{prop:ADC} says that the reason
for this characterisation
is that the unique basis of $\mathbb{R}^\infty$ projecting to bases of $\mathbb{R}^n$ in which the
operators $J(n)$ for $n \in \N$ are all upper-triangular with alternating sign entries is the columns of the infinitely extended
Pascal's triangle; in turn this holds because $J$ induces an induces a degree preserving action on polynomials, as in Lemma~\ref{lemma:binomialNew}(iii).

The trigonometric example in~\S\ref{subsec:trigonometricExample} and Theorem~\ref{thm:ctsLP} show that this
result has no immediate continuous analogue.
Indeed, the operator $L_H : \H_L \rightarrow \H_L$ corresponding to left-multiplication by the down-step matrix in the trigonometric example
in~\S\ref{subsec:trigonometricExample} is 
\[ L_H(g) = \pi \int_0^x \frac{\sin \pi z}{1-\cos \pi x} g(z) \, \mathrm{d}z. \]
As seen after Lemma~\ref{lemma:cosAlt}, 
the operator $L_P$ is represented by an (infinite) upper-triangular matrix in its
action on these trigonometric eigenfunctions, with diagonal entries $(-1)^d/(d+1)$. 
By Lemma~\ref{lemma:ctsHtransform} and Theorem~\ref{thm:ctsLP} 
the analogous 
result holds replacing $\cos d\pi x$ with the monomials $x^d$ for the right-multiplication
operator for the $\kappac{a}{b}$-weighted involutive walk. (The Jacobi function $g_d(2x-1)$ defined
on $[0,1]$ has degree $d$,
so is a linear combination of $1,x, \ldots, x^d$.)
Therefore both settings have the continuous analogue of the global anti-diagonal eigenvalue property, but
the eigenfunctions are very different.

%
%

\def\cprime{$'$} \def\Dbar{\leavevmode\lower.6ex\hbox to 0pt{\hskip-.23ex
  \accent"16\hss}D} \def\cprime{$'$}
\providecommand{\bysame}{\leavevmode\hbox to3em{\hrulefill}\thinspace}
\renewcommand{\MR}[1]{\relax }
\providecommand{\MRhref}[2]{%
  \href{http://www.ams.org/mathscinet-getitem?mr=#1}{#2}
}
\providecommand{\href}[2]{#2}

\end{document}